\documentclass[11pt]{amsart}
\usepackage{amssymb,mathrsfs,graphicx,enumerate,color}
\usepackage{amsthm,amsfonts,amssymb,epsfig,graphics,amsmath,amsbsy,enumerate}
\usepackage{colortbl}
\definecolor{black}{rgb}{0.0, 0.0, 0.0}
\definecolor{red}{rgb}{1.0, 0.5, 0.5}

\title[Uniqueness of 3D planar contact discontinuity]{Uniqueness of a planar contact discontinuity for 3D compressible Euler system in a class of zero dissipation limits from Navier-Stokes-Fourier system}

\author[Kang]{Moon-Jin Kang}
\address[Moon-Jin Kang]
{ Department of Mathematical Sciences, \newline
Korea Advanced Institute of
Science and Technology \\
 Daejeon 34141, Korea}
 \email{moonjinkang@kaist.ac.kr}

\author[Vasseur]{Alexis F. Vasseur}
\address[Alexis F. Vasseur]{\newline Department of Mathematics, \newline The University of Texas at Austin, Austin, TX 78712, USA}
\email{vasseur@math.utexas.edu}

\author[Wang]{Yi Wang}
\address[Yi Wang]{\newline Institute of Applied Mathematics, AMSS, CAS, Beijing 100190, P. R. China
\newline
and School of Mathematical Sciences, University of Chinese Academy of Sciences,
\newline Beijing 100049, P. R. China}
\email{wangyi@amss.ac.cn}

\newtheorem{theorem}{Theorem}[section]
\newtheorem{lemma}{Lemma}[section]

\newtheorem{proposition}{Proposition}[section]
\newtheorem{remark}{Remark}[section]

\newtheorem{definition}{Definition}[section]

\newcommand{\bbr}{\mathbb R}

\newcommand{\bbn}{\mathbb N}

\newcommand{\bbt} {\mathbb T}

\numberwithin{figure}{section}
\newcommand{\beq}{\begin{equation}}
\newcommand{\eeq}{\end{equation}}
\newcommand{\bsp}{\begin{split}}
\newcommand{\esp}{\end{split}}

\def\eps{\varepsilon }

\newcommand\adots{\mathinner{\mkern2mu\raise1pt\hbox{.}
\mkern3mu\raise4pt\hbox{.}\mkern1mu\raise7pt\hbox{.}}}

\renewcommand{\div}{{\rm div}}

\def\charf {\mbox{{\text 1}\kern-.30em {\text l}}}

\newcommand{\br}{\bar\rho}
\newcommand{\bt}{\bar\theta}

\newcommand{\weakto}{\rightharpoonup}

\begin{document}
\bibliographystyle{plain}

\date{\today}

\subjclass[2010]{}

\thanks{\textbf{Acknowledgment.}  M.-J. Kang was partially supported by the NRF-2019R1C1C1009355.
A. Vasseur was partially supported by the NSF grant: DMS 1614918. Y. Wang is supported by NSFC grants No. 12090014 and 11688101.
}

\begin{abstract}
We prove the stability of a planar contact discontinuity without shear, a family of  special discontinuous solutions for the three-dimensional full Euler system, in the class of vanishing dissipation limits of the corresponding Navier-Stokes-Fourier system. We also show that solutions of the Navier-Stokes-Fourier system converge to the planar contact discontinuity when the initial datum converges to the contact discontinuity itself. This implies the uniqueness of the planar contact discontinuity in the class that we are considering. Our results give an answer to the open question, whether the planar contact discontinuity is unique for the multi-D compressible Euler system. Our proof is based on the relative entropy method, together with the theory of $a$-contraction up to a shift and our new observations on the planar contact discontinuity.
\end{abstract}

\maketitle \centerline{\date}

\tableofcontents

\section{Introduction}
\setcounter{equation}{0}

We consider the Navier-Stokes-Fourier system in three dimensions with periodic boundary: for any $x=(x_1,x_2,x_3)\in\bbt^3$, $t\geq 0$,
\begin{align}
\begin{aligned}\label{NSF}
\left\{ \begin{array}{ll}
        \partial_t \rho+ \div_x (\rho u) =0,\\
       \partial_t  (\rho u)+\div_x (\rho u\otimes u) + \nabla_x p = \nu\div_x\mathbb{S}, \\
         \partial_t \left(\rho \left(\frac{|u|^2}{2} +e\right)\right) +\div_x \left(\left( \rho \left(\frac{|u|^2}{2} +e\right)+ p\right)u \right)  = \div_x (\kappa \nabla_x\theta) +\nu\div_x (\mathbb{S}u),\end{array} \right.
\end{aligned}
\end{align}
where  the functions $\rho=\rho(t,x), u(t,x)=(u_1,u_2,u_3)^\top(t,x), e=e(t,x), \theta=\theta(t,x)$ and $p=p(t,x)$ represent respectively the fluid density, velocity, specific internal energy, absolute temperature and the pressure.

The aim of this paper is to investigate the uniqueness, and stability of special discontinuous solutions, known as contact discontinuities without dissipation, of the associated Euler equation (with $\nu=0$ and $\kappa=0$). 
The study is based on a careful study of the zero dissipation limit of the Navier-Stokes-Fourier system, for vanishing viscosities ($\nu\to 0$) and heat conductivities ($\kappa\to0$) . 
We make the following assumptions on the physical system \eqref{NSF}.

Assume that the viscous stress tensor  $\mathbb{S}$, with the coeffient $\nu>0$, is given by
\[
\mathbb{S} = \mu(\theta) (\nabla u +(\nabla u)^\top) +\lambda(\theta) (\div u) Id_{3\times3},
\]
where $(\nabla u)^\top$ denotes the transpose of the matrix $\nabla u$, and $Id_{3\times3}$ represents the $3\times3$ identity matrix.
We assume that $ \mu(\theta)$ and $\lambda(\theta) $  depend linearly on $\theta$, that is:
\beq\label{ass-mu}
 \mu(\theta)= \mu_1\theta,\quad  \lambda(\theta)= \lambda_1\theta,
\eeq
where $\mu_1$ and $\lambda_1$ are both constants satisfying the physical constraints $\mu_1>0$ and $2\mu_1+3\lambda_1> 0.$
By the Fourier laws, the heat flux is given by $-\kappa\nabla_x \theta$ in $\eqref{NSF}_3$ with $\kappa>0$ denoting the heat-conductivity coefficient. Here we assume that both the parameters $\nu$ and  $\kappa$ are the positive vanishing coefficients.

The pressure $p$ is a  function of $\rho$ and $\theta$ of the form:
\beq\label{pressure}
p(\rho,\theta)=R\rho\theta + p_e(\rho),\quad p_e(\rho)=a\rho^\gamma.
\eeq
The first part of the pressure coincides with the ideal gas laws, while the second part $p_e$ is an elastic pressure (sometimes called cold pressure)
 proportional to the isentropic pressure of the ideal gas, with the adiabatic constant $ \gamma> 1$ and both $R$ and $a$ are positive constants.


We assume that the specific internal energy $e$ has the following form:
\beq\label{e}
e(\rho,\theta)=\frac{a\rho^{\gamma-1}}{\gamma-1} + Q_1(\theta), \quad Q_1(\theta)=\left\{ \begin{array}{ll}
       \frac{R}{\gamma-1}\theta \quad &\mbox{if } \theta\ge\theta_*\\
       c_*  \theta^2 \quad &\mbox{if } 0<\theta<\theta_*,  \end{array} \right.
\eeq
where $\theta_*$ is some small positive constant, and the positive constant $c_*$ is chosen  such that $Q_1$ is Lipschitz in $\theta$. Note that  $Q_1'(\theta)>0$ for all $\theta>0$. The first term of  \eqref{e} is the energy associated to the elastic pressure of \eqref{pressure}. For $\theta>\theta_*$, \eqref{e} is consistent with the ideal gas law, and together with  \eqref{pressure}, ensures that  
\beq\label{ideal}
p=(\gamma-1) \rho e,\quad\mbox{for }~\theta\ge\theta_*. 
\eeq

The definition of $Q_1$ for extremely cold temperature $\theta<\theta_*$ imposes  the validity of the third law of thermodynamic, which states that the entropy cannot blow up when the temperature approaches the absolute zero.
The validity of the third law of thermodynamic is important in our study. Note that it requires to deviate from the ideal gas dynamics  at least  for very cold temperatures (see below \eqref{def-s1}).


We can now compute the entropy functional of the system, following the 
second law of thermodynamics:
\begin{equation}\label{2nd}
\theta ds=pd(\frac1\rho)+de.
\end{equation}
The entropy takes the form
\beq\label{defs}
s=-R\ln \rho+s_1(\theta),
\eeq
where $s_1(\theta)$ satisfy
\beq\label{def-s1}
s_1(\theta) = \int^\theta_0 \frac{Q_1'(z)}{z} dz=\left\{ \begin{array}{ll}
       \frac{R}{\gamma-1}\ln\theta \quad &\mbox{if } \theta\ge\theta_*\\
       2c_*  \theta \quad &\mbox{if } 0<\theta<\theta_*.  \end{array} \right.
\eeq
Note that the above state equation for the entropy $s$ satisfy the third law of thermodynamic, that is, the entropy  approaches a constant value as temperature approaches absolute zero.

If $\nu=0$ and $\kappa=0$, then the corresponding 3D compressible Euler system reads as
\begin{align}
\begin{aligned}\label{Euler}
\left\{ \begin{array}{ll}
        \partial_t \rho+ \div_x (\rho u) =0,\\
       \partial_t  (\rho u)+\div_x (\rho u\otimes u) + \nabla_x p = 0, \\
         \partial_t \left(\rho \left(\frac{|u|^2}{2} +e\right)\right) +\div_x \left(\left( \rho \left(\frac{|u|^2}{2} +e\right)+ p\right)u \right)  = 0.\end{array} \right.
\end{aligned}
\end{align}

In this paper, we aim to show that a planar contact discontinuity without shear for 3D Euler system \eqref{Euler} is stable and unique in the class of vanishing dissipation limit ($\kappa\rightarrow0+$ and $\nu\rightarrow0+$) of solutions to 3D compressible Navier-Stokes-Fourier system \eqref{NSF}.

\subsection{Main result}
First we describe a planar contact discontinuity solution to the 3D compressible Euler equations \eqref{Euler} for moderate temperature, that is, $\theta_\pm>\theta_*$.  If we consider the following Riemann initial data
\begin{align}
\begin{aligned}\label{contact0}
(\rho, u, \theta)(x,t=0)=\left\{ \begin{array}{ll}
        (\rho_-,u_-,\theta_-) \quad &  x_1<0\\
    (\rho_+, u_+, \theta_+)\quad & x_1>0,  \end{array} \right.
\end{aligned}
\end{align}
where $\rho_\pm>0,  \theta_\pm>\theta_*, u_\pm=(u_{1\pm}, 0, 0)^t$ are prescribed constants without shear for the 3D Euler system \eqref{Euler} with $x=(x_1,x_2,x_3)\in\mathbb{R}^3$, then the Riemann problem \eqref{Euler}-\eqref{contact0} admits a planar contact discontinuity solution (as a self-similar solution in $x_1$)
\begin{align}
\begin{aligned}\label{contact1}
(\rho, u, \theta)(x,t)=\left\{ \begin{array}{ll}
        (\rho_-,u_-,\theta_-) \quad & x_1<u_{1*}t,\\
    (\rho_+, u_+, \theta_+)\quad & x_1>u_{1*}t,  \end{array} \right.
\end{aligned}
\end{align}
provided that
\begin{equation}
u_{1-}=u_{1+}:=u_{1*}, \quad p_-=p_+,
\end{equation}
where $p_\pm:=R\rho_\pm\theta_\pm + p_e(\rho_\pm)$.

By Galilean invariance, we can assume that $u_{1*}=0$ without loss of generality. Then the planar contact discontinuity solution is a stationary solution. 

Therefore, we consider a planar, stationary and shifted contact discontinuity solution $(\bar\rho, \bar u, \bar \theta)$ to the 3D compressible Euler equations \eqref{Euler} on $\bbt^3$ defined as follows:
\begin{align}
\begin{aligned}\label{contact}
(\bar\rho, \bar u, \bar \theta)(x):=\left\{ \begin{array}{ll}
        (\rho_-, 0, \theta_-), \quad & 0<x_1<\frac{1}{2},\quad x=(x_1,x_2,x_3)\in\bbt^3\\
    (\rho_+, 0, \theta_+), \quad & \frac{1}{2}<x_1<1,  \end{array} \right.
\end{aligned}
\end{align}
with 
\beq\label{pbdy}
\bar p:=p_-=p_+\quad \mbox{on } \bbt^3.
\eeq
Since $\theta_\pm>\theta_*$, it follows from \eqref{pressure} and \eqref{ideal} that the contact discontinuity solution $(\bar\rho, \bar u, \bar \theta)$ satisfies the ideal gas equality:
\[
p_\pm =(\gamma-1)\rho_\pm e_\pm.
\]
Thus, the end states $\mathcal{E}_\pm$ for the total energy $\mathcal{E}=\rho (\frac{|u|^2}{2} +e)$ satisfy
\[
\mathcal{E}_- =\rho_- e_- = \frac{p_-}{\gamma-1} = \frac{p_+}{\gamma-1} = \rho_+ e_+ =: \mathcal{E}_+ .
\]
Therefore, we set
\beq\label{ebdy}
\bar{\mathcal{E}}:=\mathcal{E}_-=\mathcal{E}_+\quad \mbox{on } \bbt^3.
\eeq

The main theorem is on the uniqueness and stability of the planar contact discontinuity solution \eqref{contact} to the 3D compressible Euler equations \eqref{Euler} in the class of zero dissipation limits (as $\kappa\rightarrow0+$ and $\nu\rightarrow0+$) of admissible weak solutions to the 3D compressible Navier-Stokes-Fourier equations \eqref{NSF} in $\mathbb{T}^3$.

For statement of the main theorem, we first consider a convex function
\beq\label{def-FG0}
\mathcal{F} (y):=\left\{ \begin{array}{ll}
       0 \quad&\mbox{if } \rho_*:=\min\{\rho_-,\rho_+\} \le y\le  \max\{\rho_-,\rho_+\}=:\rho^* , \\
       (y-\rho_*)^2 \quad&\mbox{if }  y\le \rho_*, \\
       (y-\rho^*)^2 \quad&\mbox{if }  y\ge  \rho^*.
        \end{array} \right.
        \qquad
\eeq

\begin{theorem}\label{thm:main}
Consider the Euler system \eqref{Euler} with \eqref{pressure}-\eqref{e}.
Assume $\gamma>2$. Let $\bar U:= (\bar\rho, \bar\rho\bar u, \bar{\mathcal{E}})$ be a stationary contact discontinuity \eqref{contact} with \eqref{pbdy}-\eqref{ebdy}.
Let $U_0:=(\rho_0, \rho_0 u_0, \mathcal{E}_0)$ be an initial datum to the Navier-Stokes-Fourier system \eqref{NSF} such that
\[
\int_{\bbt^3} (-\rho s) (U_0|\bar U) dx <\infty ,
\]
with the relative entropy $(-\rho s) (U_0|\bar U)$ defined in \eqref{r-entropy} and
\beq\label{bdd-near}
(-\rho s) (U_0|\bar U) \in L^\infty (\Omega),\quad\mbox{for some neighborhood $\Omega$ of the plane } x_1=\frac{1}{2} .
\eeq
For any given $T>0$, let $(\rho^{\kappa, \nu}, m^{\kappa, \nu}, \mathcal{E}^{\kappa, \nu})$ be an admissible weak solution of system \eqref{NSF} on $[0,T]$ in the sense of Definition \ref{def-sol} in Section 2, where $m^{\kappa, \nu}:=\rho^{\kappa, \nu} u^{\kappa, \nu}$ and $ \mathcal{E}^{\kappa, \nu}:=\rho^{\kappa, \nu}\left(\frac{|u^{\kappa, \nu}|^2}{2} +e^{\kappa, \nu} \right)$.\\
Then, there exists a limit $(\rho, m, \mathcal{E})$ such that, up to a subsequence, as $\kappa\to 0$ and $\nu\to 0$,
\begin{align}
\begin{aligned}\label{thmlim}
&\rho^{\kappa, \nu} \weakto \rho  \quad\mbox{weakly}-* \quad\mbox{in }~L^\infty(0,T;L^2(\bbt^3)),\\
&m^{\kappa, \nu}\weakto m   \quad\mbox{weakly}\quad\mbox{in }~ L^2(0,T;L^{\frac{4}{3}}(\bbt^3))  ,\\
&\mathcal{E}^{\kappa, \nu}\weakto \mathcal{E}   \quad\mbox{weakly}-* \quad\mbox{in }~ L^\infty(0,T;\mathcal{M}(\bbt^3)),
\end{aligned}
\end{align}
where $\mathcal{M}$ denotes the space of Radon measures,
and
\begin{align}
\begin{aligned}\label{thmine}
&\sup_{[0,T]}\int_{\bbt^3}\Big(\mathcal{F}(\rho)  +\frac{|m|^2}{2\rho} \Big)dx + \|\mathcal{E}-\bar{\mathcal{E}}\|_{L^\infty(0,T;\mathcal{M}(\bbt^3))} \\
&\qquad \le C \sqrt{\int_{\bbt^3} (-\rho s) (U_0|\bar U) dx}+ C\int_{\bbt^3} (-\rho s) (U_0|\bar U) dx.
\end{aligned}
\end{align}
Furthermore, let  $(U_0^n)_{n\in\bbn}$ be a sequence of initial data such that
\[
\int_{\bbt^3} (-\rho s) (U_0^n|\bar U) dx \to 0 \quad\mbox{as } n\to \infty,
\]
then, any inviscid limit $(\rho^n, m^n, \mathcal{E}^n)$ satisfying \eqref{thmlim}-\eqref{thmine} and corresponding to $U_0^n$ satisfies that, up to a subsequence, as $n\to\infty$,
\begin{align*}
\begin{aligned}
&\rho^n \weakto \bar \rho   \quad\mbox{weakly}-* \quad\mbox{in }~L^\infty(0,T;L^2(\bbt^3)),\\
&m^n \to 0 (=\bar\rho\bar u) \quad\mbox{in }~ L^\infty(0,T;L^{\frac{4}{3}}(\bbt^3))  ,\\
&\mathcal{E}^n \to \bar{\mathcal{E}} \quad\mbox{in }~ L^\infty(0,T;\mathcal{M}(\bbt^3)).
\end{aligned}
\end{align*}
Therefore, the planar contact discontinuity to 3D Euler system is stable and unique in the class of vanishing dissipation limits of solutions to 3D Navier-Stokes-Fourier system \eqref{NSF}.
\end{theorem}

\subsection{Remarks for the main result}
$\bullet$ The existence of global solutions for \eqref{NSF} should be constructed following the theory developed by Feireisl \cite{Feireisl_b}. However, note that our assumption especially \eqref{ass-mu}, is not compatible with the hypotheses as stated in \cite{Feireisl_b}. For this reason, we leave the construction of these solutions to a future work. 
Assumption \eqref{ass-mu} is needed to get the uniform bound of $\nu \|\nabla u^{\kappa,\nu}\|^2_{L^2(0,T;L^2(\bbt^3))}$ (see \eqref{pri-v}) from the entropy dissipation \eqref{ent-ineq}, which is crucial for our asymptotic analysis. Also, we mention that the assumption on $\gamma>2$ is crucial for the uniform bound \eqref{pri-ent} in Remark \ref{rem:use}, which is useful in Section \ref{subsec:k}.    
\\

$\bullet$ It is well known that Lipschitz solutions to the Euler system \eqref{Euler} are stable and unique in the class of admissible weak (or entropic) solutions (see Dafermos \cite{Dafermos1} and DiPerna \cite{DiPerna}). However, the situation for discontinuous solutions is far more complicated. Especially, for the case of entropic shocks or shear flows, the uniqueness is usually not true. Recently, results on the non-uniqueness were obtained by the convex integration method introduced by De Lellis and Sz\'ekelyhidi \cite{DS09, DS10}. Using this method, Chiodaroli, De Lellis and Kreml \cite{CDK} showed the existence of Riemann initial data generated by an entropy shock for which there exist infinitely many bounded entropy solutions. We also refer to \cite{AKKMM,Ch,CDK,CFK,CK,FKKM,KM18,KKMM} for other related studies on the non-uniqueness of entropic shocks. 

These results have been extended by  B$\check{\rm r}$ezina, Kreml and M\'acha \cite{BKM} to the case of   planar contact discontinuity, for the case of the 2D isentropic Euler system with the Chaplygin gas pressure law.

In the case of rarefaction waves (that are discontinuous only at $t=0$), the uniqueness was proved in the class of entropic solutions to the multi-D Euler system (see \cite{ChenChen, FK15, FKV15}). The time-asymptotic stability and vanishing viscosity limit of isentropic Navier-Stokes equatios/Navier-Stokes-Fourier equations to the planar rarefaction wave of 2D/3D compressible Euler equations could be found in \cite{LW, LWW-1, LWW1, LWW2}. \\

$\bullet$ In Theorem \ref{thm:main}, we consider the simplest discontinuous solution for \eqref{Euler}. Note that this solution corresponds to the second characteristic field in the 1D setting, and is a fundamental building block in the study of small $BV$ entropy solutions (see for instance \cite{Dafermos2,Serre_book}). Our result shows its stability in the 3D setting.
Note that it is not known if this solution is unique among  the class of entropy solutions for the multi-D compressible Euler system.

 An important feature of Theorem \ref{thm:main} is that it proves the stability in the class of zero dissipation limits of the Navier-Stokes-Fourier system rather than entropy solutions of the Euler system. 
In addition, our theorem shows the convergence of solutions of \eqref{NSF} to the contact discontinuity when the initial value converges to the contact discontinuity itself. 
Note that we do not need any {\it a priori} regularity on the dissipation limits, which are  automatically obtained by the entropy bound. Especially, the limits do not need to be solutions to the Euler system.
In the 3D setting, it is not known whether global solutions of the Euler system can be constructed for a large class of initial data. This open question makes working in the large class of dissipation limits appealing. 

Replacing the notion of weak solutions to Euler by the one of inviscid limit is already interesting and important in the 1D setting. We first mention results on the uniqueness of entropy shocks in the class of entropy solutions satisfying the locally $BV$ regularity (see  Chen-Frid-Li \cite{Chen1}) or the strong trace property (see Vasseur et al. \cite{KVARMA, LV, Vasseur-2013} and Krupa \cite{Krupa2}). However, the global-in-time propagation of those regularities remains open (except for the system with  $\gamma=3$  see  \cite{Vasseur_gamma3}). 
Recently, Kang and Vasseur \cite{KV-unique19} proved the uniqueness  and stability of entropy shocks for the 1D isentropic Euler system in the class of inviscid limits of solutions to the corresponding Navier-Stokes system.
This gives an answer, for the case of entropy shock, to the  conjecture: The compressible Euler system admits a unique entropy solution in the class of vanishing viscosity solutions to the associated  compressible Navier-Stokes system (as the physical viscous system for the Euler). As a comprehensive study related to this conjecture, Bianchini and Bressan \cite{BB} obtained a global unique entropy solution to a 1D strictly hyperbolic $n\times n$ system with small BV initial datum, which is obtained from vanishing ``artificial'' viscosity limit of the associate parabolic system. 

Theorem  \ref{thm:main} answers the conjecture in the case of contact discontinuity for the full Euler system \eqref{Euler} with \eqref{pressure}-\eqref{e}.\\

$\bullet$ Our proof is based on the theory of $a$-contraction up to a shift, first developed in the one dimensional hyperbolic case  in  \cite{KVARMA,Vasseur-2013}. The main idea is the construction of a weight function (the $a$ function) which is both bounded and bounded by below, such that the discontinuous solutions enjoy a contraction property for the corresponding weighted relative entropy, up to a shift. With the exception of  the scalar case, considered by Leger in \cite{Leger}, the contraction property is usually not verified without weight (see \cite{Serre-Vasseur}). In the case of shocks, the method was extended to 1D Navier-Stokes in 
\cite{KV-unique19}, and to the inviscid limit in \cite{Kang-V-1} (see also  \cite{AKV,CKKV,CV,Kang19,Kang,Kang-V-NS17,Krupa1,Krupa2,SV_16dcds,Stokols} for other developments of the theory). 

For the 1D  study in the hyperbolic case of the contact discontinuity, the correct weighted function is explicit and given by $\theta_-$ on the left and $\theta_+$ on the right (see \cite{SV_16}). It was used to study the zero dissipation limit in 1D in this context in    \cite{VW}. For the other studies on the vanishing dissipation limit to the 1D  Riemann solution which may contain shock and rarefaction waves and contact discontinuity, one can refer \cite{HWY, HWY1, HWY2, HWWY} and references therein. Note that the limit was proved in \cite{HWWY} for the generic 1D Riemann solution. 

Our result in Theorem \ref{thm:main} can be seen as the extension of the work of \cite{VW} in the multi-D setting. It is the first application of the method for systems in multi-D (see \cite{KVW} for an application to the multi-D scalar case).

\subsection{Ideas of the proof} \label{sec:idea}

Let us denote $U=(\rho, \rho u, \mathcal{E})$ the conservative variables, and $\eta(U)=-\rho s$ the entropy. Since $\eta$ is convex, the relative entropy 
\beq\label{r-entropy}
\eta(U|\bar{U})=\eta(U)-\eta(\bar{U})-d\eta(\bar{U})\cdot(U-\bar{U})
\eeq
define a pseudo distance of $U$ to $\bar{U}$ which is locally (for $U$ and $\bar{U}$ bounded) equivalent to 
$|U-\bar{U}|^2$. 
Consider now $U(t,x)$ solution to the system \eqref{NSF}, and $\bar{U}_-$, $\bar{U}_+$ the left and right states of the given contact discontinuity.
The general idea of the theory of a-contraction with shift for the hyperbolic case ($\nu$=$\kappa$=0) is to find disjoint shifted domains for all $t$,  $\Omega_-(t)$, $\Omega_+(t)$ such that $\Omega_-(t)\cup \Omega_+(t)=\bbt^3$, and coefficients $a_->0, a_+>0$ such that 
$$
t\to a_-\int_{\Omega_-(t)}\eta(U(t,x)|\bar{U}_-)\,dx+a_+\int_{\Omega_+(t)}\eta(U(t,x)|\bar{U}_+)\,dx
$$
is decreasing in time. Note that $p_+=p_-=\bar p$, therefore
in the Lagrangian variables, the second law of thermodynamics gives that
$$
({\bf 1}_{\{x_1<1/2\}} \theta_-+{\bf 1}_{\{x_1>1/2\}} \theta_+)d s =\bar p \,d ({1/\rho}) +de
$$
is an exact form globally on $\bbt^3$. And so, at least formally, the contraction holds for 
$a_- =\theta_-$, $a_+=\theta_+$,  $\Omega_-(t)$ the set $\{x<1/2\}$ transported by the flow $\{u(s,\cdot), 0<s<t\}$, and $\Omega_+(t)$ the set $\{x\geq1/2\}$ transported by the same  flow (see \cite{SV_16}).
Still formally, considering $(\rho(t,x), \rho(t,x)u(t,x), \rho(t,x)[e(t,x)+\frac{|u(t,x)|^2}{2}])$ a solution to the Euler equation \eqref{Euler}, and  $\psi$ the solution to the transport equation:
\begin{equation}\label{transport}
\partial_t \psi+u\cdot\nabla\psi=0,\qquad \psi(0)={\bf 1}_{\{x_1>1/2\}}, 
\end{equation}
the following function 
\beq\label{psed}
\int (\theta_- (1-\psi(t,x))\eta(U(t,x)|\bar U_-)+\theta_+\psi(t,x) \eta(U(t,x)|\bar U_+))\,dx
\eeq
would be non-increasing in time.

The main obstruction to the study above is that the velocity $u$ solution to Euler system is not smooth enough to construct solution to \eqref{transport}. It also assumes some non-oscillatory behavior (strong traces) at the boundary of $\Omega_\pm(t)$, properties that are not known to exist for solutions to Euler. 

The idea is then to consider the Navier-Stokes-Fourier system instead, which provides solutions $u\in L^2(0,T;H^1(\bbt^3))$. With this regularity, the flow \eqref{transport} can be constructed. Following \cite{CV, VW}, we consider the extra viscous $\nu$ and $\kappa$ terms as source terms to be controlled. We first pass into the limit as $\kappa$ goes to zero.  For this limit,  extra regularizations have to be performed, via convolutions, both on $u$ and $\psi$. When $\kappa=0$, we pass to the limit in the regularization terms, using controls on Lions commutators (see Lemma \ref{lem:comm}). 

When the last limit $\nu$ goes to zero, the regularity on the velocity $u$ is lost, and so $\Omega_-(t)$ and $\Omega_+(t)$ can become mixed to each other. Note that 
$ \mathcal{E}_-= \mathcal{E}_+=\bar { \mathcal{E}}$, since the contact discontinuity has values in the regime where the gas verifies the ideal gas equality \eqref{ideal}. 
Therefore it can be shown in Lemma \ref{lem:newf} that  the weighted relative entropy \eqref{psed} controls uniformly 
$$
\int_{\bbt^3} \mathcal{F}(\rho)+ \widetilde{\mathcal G}(\mathcal{E}-\bar { \mathcal{E}})+\frac{\rho|u|^2}{2},
$$
where  $\widetilde{\mathcal G}$ is the convex function:
\begin{eqnarray*}
\widetilde{\mathcal G}(y)=\left\{ \begin{array}{ll}
|y|^2, & \text{for } |y|<1,\\[1mm]
 2|y|-1,   &\text{for } |y|\geq 1.
  \end{array} \right.
\end{eqnarray*}
Note that the function $\widetilde{\mathcal G}$ is equivalent to $\mathcal{G}(y)=\min \{|y|^2,|y|\}$ defined in \eqref{def-FG}, and since we are in a bounded domain, this provides a control on the $L^1$-norm of $\mathcal{E}-\bar { \mathcal{E}}$.
However the function $\mathcal F$ is degenerated, and  is  equal to zero for $\rho_*\le \rho\le \rho^*$. Consequently, \eqref{psed} controls the perturbations in both velocity and energy.
However, due to the possible mixing of the area $\Omega_-$ and $\Omega_+$,  \eqref{psed} controls only the value of the density outside $[\rho_*,\rho^*]$. This degeneracy is similar to the situation of the pressureless Euler system in \cite{Kang-V-2}. Thus, we recover the control of the density in the similar way as  \cite{Kang-V-2}, using the continuity equation. 
Indeed, if $\rho u=0$ at the limit (which is a quantity controlled by the relative entropy), then it follows from the continuity equation that $\partial_t \rho=0$, and so $\rho =\bar\rho$.\\

The rest of the paper is as follows. In Section 2, we first present the notion on admissible weak solutions to \eqref{NSF}. Then we provide the explicit form for the weighted relative entropy, and construct nonlinear functionals uniformly controlled by it. Section 3 is dedicated to the proof of the main theorem.

\section{Preliminaries}
\setcounter{equation}{0}

In this section, we first present the Definition \ref{def-sol} as mentioned in Theorem \ref{thm:main} on admissible weak solutions of \eqref{NSF}, together with useful bounds in Remark \ref{rem:use}.  Then we compute the explicit form for the weighted relative entropy, and construct nonlinear functionals controlled by it.

\subsection{Admissible weak solutions for the Navier-Stokes-Fourier system \eqref{NSF}}.

\begin{definition} \label{def-sol}
We say that $(\rho,\rho u, \mathcal E, \rho s)$ with $ \mathcal E:=\rho \left(\frac{|u|^2}{2} +e\right)$ is an admissible weak solution to the 3D compressible Navier-Stokes-Fourier system \eqref{NSF}-\eqref{e} in $[0,T]\times\mathbb{T}^3$ with an initial data $(\rho_0,u_0, \theta_0)$
 if the following holds:\\
(i) (Mass conservation equation)
For any test function $\varphi\in \mathcal{D}([0,T)\times\mathbb{T}^3)$,
\beq\label{mass-con}
\int_0^\infty\int_{\bbt^3}  \Big(\rho\partial_\tau \varphi +\rho u \cdot \nabla \varphi\Big) dxd\tau 
+\int_{\bbt^3}\rho_0(x)\varphi(x,0)dx =0,
\eeq
and
\beq\label{ini-den}
\rho \in C([0,T]; L^\gamma_{\rm weak}(\bbt^3)) .
\eeq
(ii) (Momentum conservation equation)
For any  $\vec{\varphi}\in \big[\mathcal{D}([0,T)\times\mathbb{T}^3)\big]^3$,
\beq\label{momentum-con}
\int_0^\infty\int_{\bbt^3}  \Big(\rho u\cdot\partial_\tau \vec{\varphi} +\rho u\otimes u: \nabla \vec{\varphi}+p\, {\rm div} \vec{\varphi}-\nu\mathbb{S}:\nabla\vec{\varphi}\Big) dxd\tau +\int_{\bbt^3}\rho_0u_0(x)\cdot\vec{\varphi}(x,0)dx =0.
\eeq
(iii) (Energy inequality)
\beq\label{energy-d}
\int_{\bbt^3} \mathcal{E}(x,t) dx  \leq \int_{\bbt^3} \mathcal{E}(x,0) dx <\infty  , \qquad \forall t\in[0,T],
\eeq
where $\mathcal{E}|_{t=0}$ satisfies the compatibility condition:
\[
 \mathcal{E}|_{t=0}= \rho_0 \Big(\frac{|u_0|^2}{2} +\frac{a\rho_0^{\gamma-1}}{\gamma-1} + Q_1(\theta_0) \Big) ,
\]
and
\beq\label{ini-ene}
\limsup_{t\to 0+} \int_{\bbt^3} \mathcal{E}(x,t) dx = \int_{\bbt^3} \mathcal{E}(x,0) dx .
 \eeq
(iii) (Entropy dissipation)
For any non-negative test function $\varphi\in \mathcal{D}([0,T)\times\mathbb{T}^3)$, 
\begin{align}
\begin{aligned}\label{ent-ineq}
&\int_0^\infty\int_{\bbt^3} \Big[ \rho s\varphi_\tau+\rho su\cdot\nabla\varphi- \kappa\frac{\nabla\theta}{\theta}\cdot\nabla\varphi+\kappa\frac{ |\nabla \theta|^2}{\theta^2}\varphi + \frac{\nu \mathbb{S} : \nabla_x u}{\theta}\varphi \Big] dxd\tau \\
&\qquad \qquad+\int_{\bbt^3}(\rho s)(x,0)\varphi(x,0) dx\leq 0,
\end{aligned}
\end{align}
where $(\rho s)|_{t=0}$ satisfies the compatibility condition: 
\beq\label{comp-ent}
(\rho s)|_{t=0} = -R\rho_0\ln \rho_0+\rho_0 s_1(\theta_0) ,
\eeq
and
\beq\label{ini-econ}
\limsup_{t\to0+} \int_{\bbt^3} (\rho s)(x,t) \eta(x) dx =  \int_{\bbt^3} (\rho s)(x,0) \eta(x) dx ,\quad \forall \eta\in\mathcal{D}(\bbt^3) .
\eeq

\end{definition}


\begin{remark} \label{rem:use}
If $\gamma>2$, then by \eqref{energy-d}, \eqref{ent-ineq} and \eqref{defs}, there exists a constant $C$ independent of $\kappa, \nu$ (but depending on the initial data) such that
\beq\label{pri-mass}
\| \rho\|_{L^\infty(0,T;L^2(\bbt^3))} \le C ,
\eeq
\beq\label{pri-ent}
 \| \rho s\|_{L^\infty(0,T;L^2(\bbt^3))} \le C ,
\eeq
and
\beq\label{pri-v}
 \nu \|\nabla u\|^2_{L^2(0,T;L^2(\bbt^3))} +\kappa\left\|\frac{\nabla\theta}{\theta}\right\|^2_{L^2(0,T;L^2(\bbt^3))}\le C.
\eeq
Indeed, since $\| \rho\|_{L^\infty(0,T;L^\gamma(\bbt^3))} \le C$ by \eqref{energy-d}, we have \eqref{pri-mass}. The uniform bound \eqref{pri-v} is obtained by the entropy dissipation of \eqref{ent-ineq} with \eqref{ass-mu}.
Also, note that \eqref{defs},  \eqref{def-s1} and  \eqref{energy-d} imply that
\begin{align*}
\begin{aligned}
\int_{\bbt^3} |\rho s|^2 dx &\le C\Big(\int_{\bbt^3} \rho^2|\ln\rho |^2 \mathbf{1}_{\rho>1} dx + \int_{\bbt^3} \rho^2 dx + \int_{\bbt^3} \rho^2|\ln\theta|^2  \mathbf{1}_{\theta>\theta_*}  dx \Big) \\
&\le  C\Big(\int_{\bbt^3} \rho^\gamma  dx +\int_{\bbt^3} \rho^2 dx+ \int_{\bbt^3} \rho^2 \theta^{\frac{\gamma-2}{\gamma-1}}  \mathbf{1}_{\theta>\theta_*}  dx \Big)\\
&\le  C\Big(\int_{\bbt^3} \rho^\gamma  dx +\int_{\bbt^3} \rho^2 dx+ \int_{\bbt^3} (\rho^\gamma+\rho \theta ) \mathbf{1}_{\theta>\theta_*}  dx \Big) \le C,
\end{aligned}
\end{align*}
which gives \eqref{pri-ent}. \\
Note that, to obtain \eqref{pri-ent}, we use the setting \eqref{e} and \eqref{def-s1} satisfying the third law of thermodynamics. This is the only place we use it.
\end{remark}

\subsection{Weighted relative entropy}
We define a function $\mathcal{S}:\bbr_+\to\bbr$ by
\beq\label{news}
\mathcal{S} (Q_1(\theta)) := s_1(\theta),\quad \theta>0.
\eeq
Then the function $\mathcal{S}(Q_1)$ is strictly concave in $Q_1$. Indeed, since $\mathcal{S}'  (Q_1(\theta)) Q_1'(\theta)=s_1'(\theta)$, and $s_1'(\theta) = Q_1'(\theta)/\theta$ by \eqref{def-s1}, we have
\beq\label{SQ}
\mathcal{S}'  (Q_1(\theta))=1/\theta,
\eeq
and thus $\mathcal{S}''  (Q_1(\theta))Q_1'(\theta)=-1/\theta^2 <0$, which yields $S''  (Q_1)<0$ by $Q_1'>0$.\\

Let $U:=(\rho, m, \mathcal{E})$ be a solution of \eqref{NSF} where $m=\rho u$ and $ \mathcal E=\rho \left(\frac{|u|^2}{2} +e\right)$, and $\bar U:=(\bar \rho, \bar m, \bar {\mathcal{E}})$ be the contact discontinuity \eqref{contact} where  $\bar m=0$ and $\bar {\mathcal{E}}=\bar \rho \bar e $.
As in \cite{SV_16,VW}, we will consider the relative entropy functional weighted by the temperature $\bar \theta$ connecting two different constants $\theta_-$ and $\theta_+$:
\beq\label{wrs}
\int_{\bbt^3} \bar \theta (-\rho s) (U|\bar U)  dx,
\eeq
where $(-\rho s) (U|\bar U)$ represents the relative entropy associated with the entropy $-\rho s$, defined as follows: for $\eta(U):=-\rho s$,
\beq\label{gen-e}
(\eta)(U|\bar U):=\eta(U|\bar U)= \eta(U)-\eta(\bar U) -d \eta(U)\cdot (U-\bar U).
\eeq
In general, the notation \eqref{gen-e} will be used for a given function $\eta$ throughout the paper.
Note that if $U\mapsto \eta(U)$ is strictly convex, then $\eta(U|\bar U)$ is positively definite.
Thanks to the following lemma, we see that the functional \eqref{wrs} can be written as the sum of four sub-functionals that are all positively definite.

\begin{lemma}
Let $U:=(\rho, m, \mathcal{E})$ be a solution of \eqref{NSF} with $m=\rho u$ and $\mathcal{E}= \rho\Big(\frac{|u|^2}{2} +e \Big)$. Let $\bar U:=(\bar \rho, \bar m=0, \bar {\mathcal{E}})$ be the planar contact discontinuity \eqref{contact}. Then,
\begin{align}
\begin{aligned}\label{rel-def}
\bar \theta (-\rho s) (U|\bar U)   = \bt R (\rho\ln\rho) (\rho|\br)  + \bt\rho (-\mathcal{S})(Q_1(\theta)| Q_1(\bt))  + \frac{p_e(\rho|\br) }{\gamma-1}  +\frac{|m|^2}{2\rho}.
\end{aligned}
\end{align}
Here, since all of $\rho\mapsto\rho\ln\rho$, $\rho\mapsto p_e(\rho)$, $Q_1\mapsto (-\mathcal{S})(Q_1)$ and $(\rho, m) \mapsto \frac{|m|^2}{2\rho}$ are strictly convex, all the terms on the right-hand side of \eqref{rel-def} are positively definite.
\end{lemma}
\begin{proof}
First of all, by the definitions \eqref{defs} and \eqref{news}, we have
\beq\label{imss}
(-\rho s)(U|\bar U) = R \rho\ln\rho -\rho \mathcal{S} (Q_1(\theta))  -\Big(R \br\ln\br -\br \mathcal{S} (Q_1(\bt)) \Big)  - d(-\rho s)(\bar U)\cdot (U-\bar U).
\eeq
To compute the last term above, we first note that
\beq\label{ints}
d(-\rho s)(\bar U) = \partial_\rho(-\rho s)(\bar U) d\rho + \partial_m(-\rho s)(\bar U) d m+ \partial_{\mathcal{E}} (-\rho s) (\bar U)  d \mathcal{E} = - s (\bar U) d\rho+ \br d(-s)(\bar U) .
\eeq
To compute $d(-s)(\bar U)$ above, we use the thermodynamic relation \eqref{2nd} with $e=\frac{\mathcal{E}}{\rho} -\frac{|u|^2}{2}$, that is,
\[
\theta ds=  p d\Big(\frac{1}{\rho}\Big) + d\Big(\frac{\mathcal{E}}{\rho} -\frac{|u|^2}{2}\Big) = -\frac{p}{\rho^2}d\rho -\frac{\mathcal{E}}{\rho^2}d\rho +\frac{1}{\rho} d\mathcal{E} - u du.
\]
This together with $\bar u=0$ implies
\[
\bar\theta d(-s)(\bar U) = \frac{\bar p+\bar{\mathcal{E}}}{\br^2}d\rho - \frac{1}{\br} d\mathcal{E},
\]
which together with \eqref{ints} yields
\[
\bar\theta d(-\rho s)(\bar U) = \Big(\bt(- s) (\bar U)+ \frac{\bar p+\bar{\mathcal{E}}}{\br}\Big) d\rho -  d\mathcal{E},
\]
Thus,
\beq\label{ders}
\bar\theta d(-\rho s)(\bar U)\cdot (U-\bar U) = \Big(\bt(- s) (\bar U)+ \frac{\bar p+\bar{\mathcal{E}}}{\br}\Big) (\rho-\br) - (\mathcal{E}-\bar{\mathcal{E}}) .
\eeq
Therefore, we plug this into \eqref{imss} with \eqref{defs} to obtain
\begin{align*}
\begin{aligned}
\bt (-\rho s)(U|\bar U) &= \bt R \rho\ln\rho -\bt\rho\mathcal{S} (Q_1(\theta)) - \bt R \br\ln\br +\bt\br \mathcal{S} (Q_1(\bt))\\
&\quad +\bt \big(-R\ln\br +\mathcal{S} (Q_1(\bt))\big) (\rho-\br) \underbrace{-\frac{\bar p+\bar{\mathcal{E}}}{\br}(\rho-\br)  +  (\mathcal{E}-\bar{\mathcal{E}}) }_{=:J}.
\end{aligned}
\end{align*}
For $J$, we use $\bar {\mathcal{E}}=\bar\rho\bar e$ and  \eqref{ideal} with $\bar\theta\ge\theta_*$, to find that
\[
\frac{\bar p+\bar{\mathcal{E}}}{\br}= \gamma \bar e ,
\]
which yields
\[
J=- \gamma \bar e (\rho-\br)  + \rho e - \br\bar e +  \frac{|m|^2}{2\rho}.
\]
Since it follows from \eqref{pressure} and \eqref{e} that 
\[
\rho e = \frac{p_e(\rho)}{\gamma-1} + \rho Q_1(\theta),
\]
and especially,
\[
\gamma \bar e = \frac{p_e'(\br)}{\gamma-1} + \gamma Q_1(\bar\theta),
\]
we use \eqref{e} and \eqref{pressure} to have
\begin{align*}
\begin{aligned}
J&= - \frac{p_e'(\br)}{\gamma-1}  (\rho-\br) -\gamma Q_1(\bar\theta)(\rho-\br) + \frac{p_e(\rho)}{\gamma-1}  +\rho Q_1(\theta)-\frac{p_e(\bar\rho)}{\gamma-1}  - \bar\rho Q_1(\bar\theta)+  \frac{|m|^2}{2\rho}\\
&= \frac{1}{\gamma-1} p_e(\rho|\br) -\gamma Q_1(\bar\theta)(\rho-\br) +\rho Q_1(\theta)  - \bar\rho Q_1(\bar\theta)+  \frac{|m|^2}{2\rho}.
\end{aligned}
\end{align*}
Moreover, using the fact that $Q_1(\bar\theta) = \frac{R}{\gamma-1} \bar\theta$ by \eqref{e}, 
we have
\[
J=  \frac{1}{\gamma-1} p_e(\rho|\br) -R \bar\theta(\rho-\br) +\rho \big(Q_1(\theta) - Q_1(\bar\theta) \big)  +  \frac{|m|^2}{2\rho}.
\]
Hence, we have
\begin{align*}
\begin{aligned}
\bt (-\rho s)(U|\bar U) &= \bt R  \Big(\rho\ln\rho- \br\ln\br - (\ln\br +1) (\rho-\br) \Big) \\
&\quad -\bt \rho\Big(\mathcal{S} (Q_1(\theta))  - \mathcal{S} (Q_1(\bt)) -\frac{1}{\bt}\big(Q_1(\theta)-Q_1(\bt)\big)  \Big) +\frac{1}{\gamma-1} p_e(\rho|\br) +\frac{|m|^2}{2\rho} .
\end{aligned}
\end{align*}
Since $\mathcal{S}'(Q_1(\bt))=1/\bt$ by \eqref{SQ}, and $\frac{d(\rho\ln\rho)}{d\rho}=\ln\rho +1$, we have the desired representation \eqref{rel-def}.
\end{proof}

\subsection{Nonlinear functionals controlled by the weighted relative entropy}
We recall the convex function \eqref{def-FG0} and define a nonnegative function $\mathcal{G} $ as follows: 
\beq\label{def-FG}
\mathcal{G} (y):=\min\{|y|, y^2\}.
\eeq

\begin{lemma}\label{lem:newf}
Let $\bar U:=(\bar \rho, \bar m=0, \bar{\mathcal{E}})$ be the contact discontinuity \eqref{contact}, and $U_\pm:=(\rho_\pm, \rho_\pm u_\pm=0, \mathcal{E}_\pm=\bar{\mathcal{E}})$. Then, there exists a constant $C>0$ such that for any solution $U:=(\rho, m, \mathcal{E})$  of \eqref{NSF},
\begin{align}
\begin{aligned}\label{new-left}
\mathcal{F}(\rho) + \mathcal{G}(\mathcal{E}-\bar{\mathcal{E}}) +\frac{|m|^2}{2\rho}
\le C \min \Big\{\theta_- (-\rho s) (U|U_-) ,\theta_+ (-\rho s) (U|U_+)    \Big\} .
\end{aligned}
\end{align}
\end{lemma}
\begin{proof}
First, since $\gamma>2$, there exists a constant $C>0$ such that $|\rho-\rho_\pm|^2 \le C p_e(\rho|\rho_\pm)$ for all $\rho\ge 0$, which together with \eqref{def-FG0} implies
\[
\mathcal{F}(\rho) \le C\min\{p_e(\rho|\rho_+), p_e(\rho|\rho_-)\},\qquad \forall \rho\ge 0.
\]
Thus, it follows from \eqref{rel-def} that
\[
\mathcal{F}(\rho) +\frac{|m|^2}{2\rho} \le C \min \Big\{\theta_- (-\rho s) (U|U_-) ,\theta_+ (-\rho s) (U|U_+)  \Big\}.
\]
Therefore, it remains to prove that
\beq\label{G-claim}
 \mathcal{G}(\mathcal{E}-\bar{\mathcal{E}}) \le C\min \Big\{\theta_- (-\rho s) (U|U_-) ,\theta_+ (-\rho s) (U|U_+)  \Big\}.
\eeq
Indeed, using the following Lemma \ref{lem:gene} together with the convex open set $\Omega=\bbr_+\times\bbr\times\bbr$ and the strict convex function $\eta=-\rho s$, we find that for each $\bar{\mathcal{E}} = \mathcal{E}_\pm$,
\begin{align*}
 \mathcal{G}(\mathcal{E}-\bar{\mathcal{E}}) &= \min \big\{ |\mathcal{E}-\bar{\mathcal{E}}|,  |\mathcal{E}-\bar{\mathcal{E}}|^2 \big\}
 \le \min \big\{ |U-U_\pm|,  |U-U_\pm|^2 \big\} \\
 & \le C  (-\rho s) (U|U_\pm) \le C \theta_\pm (-\rho s) (U|U_\pm),
\end{align*}
which together with $\bar{\mathcal{E}} = \mathcal{E}_- =\mathcal{E}_+$ implies \eqref{G-claim}.
\end{proof}

\begin{lemma}\label{lem:gene}
Let $\Omega$ is a convex open subset of $\bbr^n$. Let $\eta:\Omega\to \bbr$ be a strictly convex function. The relative function $\eta(\cdot|\cdot)$ (defined as in \eqref{gen-e}) satisfies the following:\\
For any $z_0\in\Omega$, there exists a constant $C>0$ such that
\[
\eta(z|z_0)\ge C \min\big\{ |z-z_0|, |z-z_0|^2 \big\} ,\quad\forall z\in\Omega.
\]
\end{lemma}
\begin{proof}
Since $\Omega$ is open, for any fixed $z_0\in\Omega$ there exists a constant $r>0$ such that
\[
\{z~|~|z-z_0|=r\}\subset \Omega.
\]
For any vector $v\in\bbr^n$ such that $|v|=r$, we define a non-negative function $f_v:\bbr\to\bbr_+$ by
\[
f_v(t):= \eta(z_0+t v ~|~ z_0).
\]
Then,
\[
f_v'(t) = \Big( d\eta(z_0+ t v) - d\eta(z_0)  \Big) \cdot v.
\]
Since $\eta$ is strictly convex on the convex set $\Omega$, $f_v''(t)= v \big(d^2 \eta)(z_0+tv)v^t>0$ on the interval $I=\{t~|~z_0+tv\in\Omega\}$, which implies that
$f_v$ is strictly convex on $I$.\\
Since $f_v'(0)=0$, we have
\[
\Big( d\eta(z_0+ v) - d\eta(z_0)  \Big) \cdot v = f_v'(1)>0.
\]
Moreover, since $v\mapsto\Big( d\eta(z_0+ v) - d\eta(z_0)  \Big) \cdot v$ is continuous on the compact set $|v|=r$, there exists a constant $C(r)$ such that
\[
f_v'(1) \ge C(r)\quad \forall v~ \mbox{with}~|v|=r.
\]
Thus, for all $v$ with $|v|=r$, 
\[
 f_v' (t) \ge  C(r)\quad\forall ~t\ge 1,
\] 
which implies that for all $v$ with $|v|=r$,
\[
\forall t\ge 2,\quad  \quad f_v (t) = f_v(1) +\int_1^t f_v' \ge  \int_1^t f_v' \ge C(r) (t-1) \ge \frac{C(r)}{2} t.
\] 
Now, for any $z\in\Omega$ with $|z-z_0|\ge 2r$, by putting $v=\frac{z-z_0}{|z-z_0|}r$ and $t=\frac{|z-z_0|}{r}$ above, we have
\[
\eta(z|z_0) = f_v(t) \ge  \frac{C(r)}{2r} |z-z_0|.
\]
On the other hand, the strict convexity of $\eta$ and the definition of the relative function imply that there exists a constant $C_r>0$ such that for any $z\in\Omega$ with $|z-z_0|\le 2r$,
\[
\eta(z|z_0) \ge C_r |z-z_0|^2.
\]
Those two estimates imply the desired estimate.
\end{proof}

\section{Proof of Theorem \ref{thm:main}}
\setcounter{equation}{0}
As explained in Section \ref{sec:idea}, we will construct the $\psi$ function transported by the velocity field, with regularizations.

\subsection{Construction of shift functions}
For the velocity $u^{\kappa,\nu}$ being the weak solution to \eqref{NSF}, we first consider a family of spatial-mollifications of $u$ as
\[
\bar u^{\kappa,\nu}_\delta(x,t) := \int_{\bbr^3} u^{\kappa,\nu}(x-y,t) \eta_\delta (y) dy,
\]
where $\eta_\delta$ denotes the mollifier defined by
\beq\label{def-time}
\eta_\delta(x):=\frac{1}{\delta^3}\eta \big(\frac{|x|}{\delta}\big)\quad\mbox{for any }\delta>0,
\eeq
for a non-negative smooth function $\eta:\bbr\to\bbr$ such that  $\int_{\bbr}\eta=1$.\\
Note that since $u$ is periodic in $x$, so is $\bar u^{\kappa,\nu}_\delta$. Indeed, for each $i=1,2,3$,
\[
\bar u^{\kappa,\nu}_\delta(x+e_i,t) =  \int_{\bbr^3}\bar u^{\kappa,\nu}(x+e_i-y,t) \eta_\delta (y) dy= \int_{\bbr^3} \bar u^{\kappa,\nu}(x-y,t) \eta_\delta (y) dy =\bar u^{\kappa,\nu}_\delta(x,t),
\]
where and in the sequel $e_i~(i=1,2,3)$ is the standard unit vector.
For the smooth velocity field $\bar u^{\kappa,\nu}_\delta$, we define $\psi^{\kappa,\nu}_\delta$ as the unique solution to the following transport equation:
\begin{equation}\label{psi-eq}
\left\{
\begin{array}{lll}
      \displaystyle \partial_t \psi + \bar u^{\kappa,\nu}_\delta \cdot \nabla\psi =0,\\
       \displaystyle \psi(x,0) =: \psi_0(x) =\left\{ \begin{array}{ll} 0, \quad \mbox{if } 0<x_1<\frac{1}{2},  \\1, \quad\mbox{if }  \frac{1}{2}<x_1<1, \end{array} \right.  \quad x=(x_1,x_2,x_3)\in\bbt^3 .
\end{array}
\right.
\end{equation}
For any fixed $(x,t)\in\bbr^3\times[0,T]$, we define a characteristic curve $X(\tau;x,t)$ generated by $u^{\kappa,\nu}_\delta$, passing through $x$ at $\tau=t$ as follows:
\begin{equation}\label{char}
\left\{
\begin{array}{lll}
      \displaystyle
       \frac{d}{d\tau} X(\tau;x,t) =\bar u^{\kappa,\nu}_\delta (X(\tau;x,t),\tau),\\[3mm]
       \displaystyle X(\tau=t;x,t) =x.
\end{array}
\right.
\end{equation}
Note that since $\bar u^{\kappa,\nu}_\delta$ is smooth in $x$ and $\|\bar u^{\kappa,\nu}_\delta\|_{C^1(\bbt^3)} \in L^2(0,T)$, the above ODE has a unique absolutely continuous solution $X(\tau;x,t)$ on $\tau\in[0,T]$.\\
Then, since it follows from \eqref{psi-eq} and \eqref{char} that
\beq\label{exppsi}
\psi^{\kappa,\nu}_\delta(x,t) = \psi_0(X(0;x,t)),
\eeq
it holds that
\beq\label{psi-bdd}
0\le \psi^{\kappa,\nu}_\delta(x,t)\le 1,\quad\forall (x, t)\in\mathbb{T}^3\times[0,T].
\eeq
\begin{remark}
Note that the solution $\psi^{\kappa,\nu}_\delta(x,t)$ is periodic in $x$. Indeed, 
since $\bar u^{\kappa,\nu}_\delta$ is periodic in $x\in\mathbb{T}^3$, it follows from \eqref{char} that
for each $i=1,2,3$,
\[
\frac{d}{d\tau} \big( X(\tau;x+e_i,t) -e_i \big)=\frac{d}{d\tau} X(\tau;x+e_i,t) =\bar u^{\kappa,\nu}_\delta (X(\tau;x+e_i,t),\tau) =\bar u^{\kappa,\nu}_\delta \big(X(\tau;x+e_i,t)-e_i, \tau\big).
\]
Thus, we find that for each $i=1,2,3$,
\begin{equation*}
\left\{
\begin{array}{lll}
      \displaystyle \frac{d}{d\tau} \big( X(\tau;x+e_i,t) -e_i\big) =\bar u^{\kappa,\nu}_\delta \big(X(\tau;x+e_i,t)-e_i, \tau\big) ,\\[3mm]
       \displaystyle \big( X(\tau;x+e_i,t) -e_i\big)|_{\tau=t} =x.
\end{array}
\right.
\end{equation*}
Therefore, by the uniqueness of solutions to \eqref{char}, for each $i=1,2,3$,
\[
X(\tau;x,t)= X(\tau;x+e_i,t) -e_i,\quad\forall \tau\in [0,T].
\]
This together with the periodicity of $\psi_0$ implies
\begin{align*}
\begin{aligned}
  \psi^{\kappa,\nu}_\delta(x+e_i,t)=\psi_0 (X(0;x+e_i,t)) = \psi_0 (X(0;x+e_i,t)-e_i) = \psi_0(X(0;x,t)) =\psi^{\kappa,\nu}_\delta(x,t) .
\end{aligned}
\end{align*}
Hence, $\psi^{\kappa,\nu}_\delta(x,t)$ is periodic in $x\in\mathbb{T}^3$.
\end{remark}

Since the solution $\psi^{\kappa,\nu}_\delta$ to the equation \eqref{psi-eq} is still discontinuous, we consider a family of mollifications of $\psi^{\kappa,\nu}_\delta$ as
\beq\label{psi-mol}
\bar \psi^{\kappa,\nu}_{\delta,\eps}(x,t) := \int_{\bbr^3} \psi^{\kappa,\nu}_\delta(x-y,t) \eta_\eps (y) dy,
\eeq
where $\eta_\eps$ denotes the mollifier defined by in \eqref{def-time} with the parameter $\delta$ replaced by another parameter $\eps$.\\
Note that by \eqref{psi-bdd},
\beq\label{psi-be}
0\le \bar \psi^{\kappa,\nu}_{\delta,\eps}(x,t)\le 1,\quad\forall (x, t), ~\forall \eps>0.
\eeq

\subsection{Weighted entropy inequality}
\begin{lemma}
Let $U:=(\rho, m, \mathcal{E})$ be a solution of \eqref{NSF} in the sense of Definition \ref{def-sol}, and $\Psi:=\theta_- \big(1-\bar \psi^{\kappa,\nu}_{\delta,\eps}\big) +\theta_+ \bar \psi^{\kappa,\nu}_{\delta,\eps}$. Then,
\begin{align}
\begin{aligned}\label{went1}
&\int_{\bbt^3} (-\rho s)(U(x,t)) \Psi (x,t) dx-\int_{\bbt^3} (-\rho s)(U_0)\Psi  (x,0) dx \\
&\quad  +\min\{\theta_+,\theta_-\} \int_0^t \int_{\bbt^3} \Big( \kappa\frac{ |\nabla \theta|^2}{\theta^2} + \frac{\mathbb{S}:\nabla u}{\theta} \Big)(x,\tau)  dxd\tau \\
&\quad + \kappa (\theta_+-\theta_-) \int_0^t \int_{\bbt^3} \frac{\nabla\theta}{\theta} \cdot \nabla\bar \psi^{\kappa,\nu}_{\delta,\eps}  dxd\tau \\
&\quad+(\theta_+-\theta_-)\int_0^t\int_{\bbt^3}  (\rho s)(U(x,\tau))\Big( \partial_\tau \bar \psi^{\kappa,\nu}_{\delta,\eps} + u^{\kappa,\nu}\cdot\nabla \bar \psi^{\kappa,\nu}_{\delta,\eps}\Big) (x,\tau)  dxd\tau \leq0
\end{aligned}
\end{align}
\end{lemma}
\begin{proof}
Let $\phi:\bbr\to\bbr$ be a non-negative smooth function such that $\phi (s)=\phi(-s)$, $\int_{\bbr}\phi=1$ and $supp\ \phi = [-1,1]$, and let
\[
\phi_\zeta(s):=\frac{1}{\zeta}\phi \big(\frac{s-\zeta}{\zeta}\big)\quad\mbox{for any } \zeta>0.
\]
Then for a given $t\in(0,T)$, and any $\zeta<t/2$, we define a non-negative smooth function
\beq\label{time-mol}
\varphi_{t, \zeta} (s) := \int_0^s \Big(\phi_\zeta (z) -\phi_\zeta (z- t)  \Big) dz.
\eeq

For test functions of the entropy inequality \eqref{ent-ineq}, we consider a family of non-negative functions
\[
\vec{\varphi}(x,\tau) :=\Psi (x,\tau) \varphi_{t, \zeta} (\tau),
\]
where $\Psi :=\theta_- \big(1-\bar \psi^{\kappa,\nu}_{\delta,\eps} \big) +\theta_+ \bar \psi^{\kappa,\nu}_{\delta,\eps}$.\\
Then, it follows from \eqref{ent-ineq} and $\varphi_{t, \zeta} (0)=0$ that
\[
 \int_0^\infty\int_{\bbt^3} (-\rho s)(U(x,\tau))  \Psi  (x,\tau) \big(-\varphi_{t, \zeta} '(\tau) \big) dxd\tau +J_\zeta^1+J_\zeta^2+J_\zeta^3\le 0,
\]
where
\begin{align*}
\begin{aligned}
& J_\zeta^1:= \int_0^\infty \int_{\bbt^3} \Big( \kappa\frac{ |\nabla \theta|^2}{\theta^2} + \nu\frac{\mathbb{S} : \nabla_x u}{\theta} \Big)(x,\tau) \Psi  (x,\tau)  \varphi_{t, \zeta} (\tau) dxd\tau , \\
&J_\zeta^2:= - \int_0^\infty\int_{\bbt^3} \kappa \frac{\nabla\theta}{\theta}  \cdot \nabla \Psi  (x,\tau)  \varphi_{t, \zeta} (\tau)  dxd\tau , \\
&J_\zeta^3:= - (\theta_+-\theta_-) \int_0^\infty \int_{\bbt^3}  (-\rho s)(U(x,\tau))\Big( \partial_\tau\bar \psi^{\kappa,\nu}_{\delta,\eps} + u\cdot\nabla \bar \psi^{\kappa,\nu}_{\delta,\eps} \Big) (x,\tau)  \varphi_{t, \zeta} (\tau)  dxd\tau .
\end{aligned}
\end{align*}
First, since $\Psi \in C([0,T];C^\infty(\bbt^3))$ and $\varphi_{t, \zeta} '(\tau) =  \phi_\zeta (\tau)-\phi_\zeta (\tau- t)$,
\eqref{comp-ent} and \eqref{ini-econ} implies that
as $\zeta\to0$,
\[
 \int_0^\infty\int_{\bbt^3} (-\rho s)(U(x,\tau)) \Psi  (x,\tau) \big(-\varphi_{t, \zeta} '(\tau) \big) dxd\tau
 \to \int_{\bbt^3} (-\rho s)(U(x,t)) \Psi  (x,t) dx-\int_{\bbt^3} (-\rho s)(U_0)\Psi  (x,0) dx.
\]
Likewise, as $\zeta\to0$,
\begin{align*}
\begin{aligned}
& J_\zeta^1\to \int_0^t \int_{\bbt^3} \Big( \kappa\frac{ |\nabla \theta|^2}{\theta^2} + \nu\frac{\mathbb{S} : \nabla u}{\theta} \Big)(x,\tau) \Psi  (x,\tau) dxd\tau =: J_1 , \\
&J_\zeta^2\to - \int_0^t \int_{\bbt^3} \kappa \frac{\nabla\theta}{\theta}  \cdot \nabla \Psi  (x,\tau)    dxd\tau =: J_2, \\
&J_\zeta^3\to - (\theta_+-\theta_-) \int_0^t \int_{\bbt^3}  (-\rho s)(U(x,\tau)) \Big( \partial_\tau\bar \psi^{\kappa,\nu}_{\delta,\eps} + u\cdot\nabla \bar \psi^{\kappa,\nu}_{\delta,\eps} \Big) (x,\tau)  dxd\tau .
\end{aligned}
\end{align*}
In particular, using the assumptions \eqref{ass-mu} together with \eqref{psi-be}, we can show 
\begin{align*}
\begin{aligned}
J_1& =  \int_0^t \int_{\bbt^3} \Big( \kappa\frac{ |\nabla \theta|^2}{\theta^2} + \nu |\nabla u|^2 \Big)(x,\tau) \Psi  (x,\tau) dxd\tau\\
&\quad \ge \min\{\theta_+,\theta_-\}\int_0^t \int_{\bbt^3} \Big( \kappa\frac{ |\nabla \theta|^2}{\theta^2} + \nu |\nabla u|^2 \Big)(x,\tau)  dxd\tau.
\end{aligned}
\end{align*}
Moreover, since
\[
J_2 =  \kappa (\theta_+-\theta_-) \int_0^t \int_{\bbt^3} \frac{\nabla\theta}{\theta} \cdot \nabla\bar \psi^{\kappa,\nu}_{\delta,\eps}  dxd\tau,
\]
we have the desired estimate.
\end{proof}

\subsection{Weighted relative entropy inequality}

\begin{lemma}\label{lem:wr}
Let $\bar U:=(\bar \rho, \bar m, \bar{\mathcal{E}})$ be the planar contact discontinuity \eqref{contact}. Then,
there exist constants $C>0$, $C_1, C_2$ such that for a solution $U:=(\rho, m, \mathcal{E})$  of \eqref{NSF} in the sense of Definition \ref{def-sol}, we have
\begin{align}
\begin{aligned}\label{went-ineq}
&\int_{\bbt^3} \Big[ \big(1-\bar \psi^{\kappa,\nu}_{\delta,\eps} \big) \theta_- (-\rho s) (U|U_-)  +\bar \psi^{\kappa,\nu}_{\delta,\eps}  \theta_+ (-\rho s) (U|U_+)  \Big](x,t) dx\\
&\qquad  +\min\{\theta_+,\theta_-\}  \int_0^t \int_{\bbt^3} \Big( \kappa\frac{ |\nabla \theta|^2}{\theta^2} + \nu |\nabla u|^2 \Big)(x,\tau)  dxd\tau\\
&\quad \le C\eps+ C\int_{\bbt^3} (-\rho s) (U_0|\bar U) dx +\kappa C \int_0^t \int_{\bbt^3} \frac{|\nabla\theta|}{\theta} |\nabla \bar \psi^{\kappa,\nu}_{\delta,\eps}|   dxd\tau \\
&\qquad+\int_0^t\int_{\bbt^3} \Big(C_1 \rho + C_2 (-\rho s)(U) \Big)\Big( \partial_\tau\bar \psi^{\kappa,\nu}_{\delta,\eps} + u^{\kappa,\nu}\cdot\nabla \bar \psi^{\kappa,\nu}_{\delta,\eps} \Big) (x,\tau)  dxd\tau .
\end{aligned}
\end{align}
\end{lemma}
\begin{proof}
First, using the definition of the relative functional, we have
\begin{align*}
\begin{aligned}
L(t)&:=\int_{\bbt^3} \Big[ \big(1-\bar \psi^{\kappa,\nu}_{\delta,\eps} \big) \theta_- (-\rho s) (U|U_-)  +\bar \psi^{\kappa,\nu}_{\delta,\eps}  \theta_+ (-\rho s) (U|U_+)  \Big](x,t) dx\\
&=\int_{\bbt^3} \Psi  (x,t) (-\rho s)(U(x,t))  dx \\
&\quad-\underbrace{\int_{\bbt^3} \big(1-\bar \psi^{\kappa,\nu}_{\delta,\eps} \big) \Big(  \theta_- (-\rho s) (U_-) +  \theta_- d(-\rho s)(U_-)\cdot (U-U_-) \Big)dx}_{=:J_1}\\
&\quad - \underbrace{\int_{\bbt^3} \bar \psi^{\kappa,\nu}_{\delta,\eps} \Big(  \theta_+ (-\rho s) (U_+) +  \theta_+ d(-\rho s)(U_+)\cdot (U-U_+) \Big)dx}_{=:J_2}.
\end{aligned}
\end{align*}
Then, using \eqref{ders},
\begin{align*}
\begin{aligned}
J_1&=\int_{\bbt^3} \big(1-\bar \psi^{\kappa,\nu}_{\delta,\eps} \big) \Big[ \theta_- \rho_- (-s)(U_- ) +\Big(\theta_- (- s) ( U_- )+ \frac{ p_- + \mathcal{E}_- }{\rho_- }\Big) (\rho-\rho_- ) - (\mathcal{E}-\mathcal{E}_- )  \Big] dx\\
&=\int_{\bbt^3} \big(1-\bar \psi^{\kappa,\nu}_{\delta,\eps} \big) \Big[ \Big(\theta_- (- s) ( U_- )+ \frac{ p_- + \mathcal{E}_- }{\rho_- }\Big) \rho -p_- - \mathcal{E}  \Big] dx.
\end{aligned}
\end{align*}
Likewise,
\[
J_2= \int_{\bbt^3}\bar \psi^{\kappa,\nu}_{\delta,\eps} \Big[ \Big(\theta_+ (- s) ( U_+ )+ \frac{ p_+ + \mathcal{E}_+ }{\rho_+ }\Big) \rho -p_+ - \mathcal{E}  \Big] dx.
\]
Set
\[
\beta_-:=\theta_- (- s) ( U_- )+ \frac{ p_- + \mathcal{E}_- }{\rho_- },\qquad \beta_+:=\theta_+ (- s) ( U_+ )+ \frac{ p_+ + \mathcal{E}_+}{\rho_+ }.
\]
Thus, using $\bar p=p_-=p_+$, we have
\begin{align}
\begin{aligned}\label{LMrel}
L(t) &= \int_{\bbt^3} \Psi  (x,t) (-\rho s)(U(x,t)) dx
- \int_{\bbt^3}  \Big[  \Big( \big(1-\bar \psi^{\kappa,\nu}_{\delta,\eps} \big) \beta_-  + \bar \psi^{\kappa,\nu}_{\delta,\eps}  \beta_+ \Big) \rho - \bar p  -\mathcal{E}  \Big](x,t) dx \\
&=:M(t),\qquad \forall t\ge 0.
\end{aligned}
\end{align}
Since $ \bar \psi^{\kappa,\nu}_{\delta,\eps}\in C([0,T];C^\infty(\bbt^3))$, the initial conditions \eqref{ini-den}, \eqref{ini-ene}, \eqref{comp-ent}  and \eqref{ini-econ} imply  
\begin{align*}
\begin{aligned}
M(0)&=\lim_{t\to 0+}M(t)\\
&=\int_{\bbt^3} \Psi  (x,0) (-\rho s)(U_0) dx
- \int_{\bbt^3}  \Big[  \Big( \big(1-\bar \psi^{\kappa,\nu}_{\delta,\eps}(x,0) \big) \beta_-  + \bar \psi^{\kappa,\nu}_{\delta,\eps}(x,0)  \beta_+ \Big) \rho_0 - \bar p  -\mathcal{E}_0  \Big] dx,
\end{aligned}
\end{align*}
which together with \eqref{LMrel} yields
\begin{align*}
\begin{aligned}
M(0)&=\int_{\bbt^3} \Psi  (x,0) (-\rho s)(U_0) dx
- \int_{\bbt^3}  \Big[  \Big( \big(1-\bar \psi^{\kappa,\nu}_{\delta,\eps} \big) \beta_-  + \bar \psi^{\kappa,\nu}_{\delta,\eps}  \beta_+ \Big) \rho_0 - \bar p  -\mathcal{E}_0  \Big] dx\\
&=\int_{\bbt^3} \Big[ \big(1-\bar \psi^{\kappa,\nu}_{\delta,\eps}(x,0) \big)  \theta_- (-\rho s) (U_0|U_-)  +\bar \psi^{\kappa,\nu}_{\delta,\eps} (x,0)  \theta_+ (-\rho s) (U_0|U_+)  \Big] dx\\
&=: L_0.
\end{aligned}
\end{align*}
Thus, we have
\begin{align*}
\begin{aligned}
L(t)-L_0 &= M(t)-M(0)\\
&=\int_{\bbt^3}   \Big[  \Psi  (x,t) (-\rho s)(U(x,t))  - \Psi  (0,t) (-\rho s)(U_0)   \Big] dx \\
&\quad
- \int_{\bbt^3}  \Big[ \rho\Big( \big(1-\bar \psi^{\kappa,\nu}_{\delta,\eps} \big) \beta_-  + \bar \psi^{\kappa,\nu}_{\delta,\eps}  \beta_+ \Big)(x,t)-\rho_0\Big( \big(1-\bar \psi^{\kappa,\nu}_{\delta,\eps} \big) \beta_-  + \bar \psi^{\kappa,\nu}_{\delta,\eps}  \beta_+ \Big)(x,0) \Big]  dx \\
&\quad +  \int_{\bbt^3}  (\mathcal{E}-\mathcal{E}_0)   dx=: R_1(t)+R_2(t)+R_3(t).
\end{aligned}
\end{align*}
First, by \eqref{energy-d}, $R_3(t)\le 0$.\\
To handle $R_2(t)$, we consider the following test functions for \eqref{mass-con}:
\[
\varphi(x,\tau) :=  \Big( \big(1-\bar \psi^{\kappa,\nu}_{\delta,\eps} \big) \beta_-  + \bar \psi^{\kappa,\nu}_{\delta,\eps}  \beta_+ \Big)(x,\tau) \varphi_{t, \zeta} (\tau),
\]
where  $\varphi_{t, \zeta}$ are as in \eqref{time-mol}. Plugging the above test functions into \eqref{mass-con}, and taking $\zeta\to 0$ together with the same argument as before, we have
\[
R_2(t)=(\beta_+ - \beta_-) \int_0^t \int_{\bbt^3}  \rho(x,\tau) \Big( \partial_\tau\bar \psi^{\kappa,\nu}_{\delta,\eps} + u\cdot\nabla \bar \psi^{\kappa,\nu}_{\delta,\eps} \Big) (x,\tau)  dxd\tau
\]
Therefore, using \eqref{went1}, we have
\begin{align*}
\begin{aligned}
L(t)&\le L_0  -\min\{\theta_+,\theta_-\} \int_0^t \int_{\bbt^3} \Big( \kappa\frac{ |\nabla \theta|^2}{\theta^2} + \nu |\nabla u|^2 \Big)(x,\tau)  dxd\tau\\
&\qquad + \kappa |\theta_+-\theta_-| \int_0^t \int_{\bbt^3} \frac{|\nabla\theta|}{\theta} |\nabla \bar \psi^{\kappa,\nu}_{\delta,\eps}|   dxd\tau \\
&\qquad +\int_0^t\int_{\bbt^3} \Big(C_1 \rho + C_2 (-\rho s)(U) \Big)\Big( \partial_\tau\bar \psi^{\kappa,\nu}_{\delta,\eps} + u\cdot\nabla \bar \psi^{\kappa,\nu}_{\delta,\eps} \Big) (x,\tau)  dxd\tau ,
\end{aligned}
\end{align*}
where $C_1:=\beta_+-\beta_-$, $C_2:=\theta_+-\theta_-$.\\
Since $\bar\psi^{\kappa,\nu}_{\delta,\eps}|_{t=0} = \psi_0*\eta_\eps$ by \eqref{exppsi}, it follows from \eqref{bdd-near} that
\begin{align*}
\begin{aligned}
L_0 &=\int_{\bbt^3} \Big[ \big(1-( \psi_0*\eta_\eps) \big) \theta_- (-\rho s) (U_0|U_-)  +( \psi_0*\eta_\eps) \theta_+ (-\rho s) (U_0|U_+)  \Big] dx \\
&\le  C\int_{\bbt^3} (-\rho s) (U_0|\bar U) dx + C\int_{1/2 -\eps}^{1/2 +\eps} (-\rho s) (U_0|\bar U) dx \\
&\le   C\int_{\bbt^3} (-\rho s) (U_0|\bar U) dx  + C\eps.
\end{aligned}
\end{align*}
\end{proof}

Note first that since the velocity $u$ depends on $\kappa,\nu$, $\bar \psi^{\kappa,\nu}_{\delta,\eps}$ depends on $\kappa, \delta, \eps,\nu$ by the definition \eqref{psi-eq}.
Thus, the left-hand side of the estimate \eqref{went-ineq} depends on $\kappa, \delta, \eps,\nu$.
Therefore, before performing the limit process for  \eqref{went-ineq}, we may first use Lemma \ref{lem:newf} to ensure that the left-hand side of \eqref{went-ineq} controls a new nonlinear functional independent of those parameters $\kappa, \delta, \eps,\nu$. Indeed, using Lemmas \ref{lem:newf} and \ref{lem:wr} together with \eqref{psi-be}, we have the following Proposition.

\begin{proposition}
Under the same hypotheses as in Lemma \ref{lem:newf}, there exist constants $C>0$, $C_1, C_2$ such that
 for a solution $U^{\kappa,\nu}:=(\rho, m, \mathcal{E})$  of \eqref{NSF},
\begin{align}
\begin{aligned}\label{went2}
&\int_{\bbt^3}\Big(\mathcal{F}(\rho^{\kappa,\nu}) + \mathcal{G}(\mathcal{E}^{\kappa,\nu}-\bar{\mathcal{E}}) +\frac{|m^{\kappa,\nu}|^2}{2\rho^{\kappa,\nu}} \Big)dx   +  \int_0^t \int_{\bbt^3} \Big( \kappa\frac{ |\nabla \theta^{\kappa,\nu}|^2}{(\theta^{\kappa,\nu})^2} + \nu |\nabla u^{\kappa,\nu}|^2 \Big)(x,\tau)  dxd\tau\\
&\quad \le C\eps + C\int_{\bbt^3} (-\rho s) (U_0|\bar U) dx + \kappa C \int_0^t \int_{\bbt^3} \frac{|\nabla\theta^{\kappa,\nu}|}{\theta^{\kappa,\nu}} |\nabla \bar \psi^{\kappa,\nu}_{\delta,\eps}|   dxd\tau \\
&\qquad+\int_0^t\int_{\bbt^3} \Big(C_1 \rho^{\kappa, \nu} + C_2 (-\rho s)(U^{\kappa, \nu}) \Big)\Big( \partial_\tau\bar \psi^{\kappa,\nu}_{\delta,\eps} + u^{\kappa,\nu}\cdot\nabla \bar \psi^{\kappa,\nu}_{\delta,\eps} \Big) (x,\tau)  dxd\tau .
\end{aligned}
\end{align}
\end{proposition}

The remaining part of the proof is dedicated to the asymptotic analysis on passing to the limits for the parameters $\kappa, \delta, \eps$ and $\nu$ in order.

\subsection{Passing to the limit as $\kappa\to 0$} \label{subsec:k}

First of all, since Young's inequality yields
\[
 \kappa C \int_0^t \int_{\bbt^3} \frac{|\nabla\theta^{\kappa,\nu}|}{\theta^{\kappa,\nu}} |\nabla \bar \psi^{\kappa,\nu}_{\delta,\eps}|   dxd\tau
 \le \frac{\kappa}{2}  \int_0^t \int_{\bbt^3}  \frac{|\nabla\theta^{\kappa,\nu}|^2}{(\theta^{\kappa,\nu})^2}  dxd\tau + \kappa C \int_0^t \int_{\bbt^3}  |\nabla \bar \psi^{\kappa,\nu}_{\delta,\eps}|^2  dxd\tau,
\]
it follows from \eqref{went2} that
\begin{align}
\begin{aligned}\label{dwent}
&\int_{\bbt^3}\Big(\mathcal{F}(\rho^{\kappa, \nu}) + \mathcal{G}(\mathcal{E}^{\kappa, \nu}-\bar {\mathcal{E}}) +\frac{|m^{\kappa,\nu}|^2}{2\rho^{\kappa,\nu}}  \Big)dx  \\
&\quad \le C\eps + C\int_{\bbt^3} (-\rho s) (U_0|\bar U) dx + J_1 + J_2 .
\end{aligned}
\end{align}
where
\begin{align*}
\begin{aligned}
&J_1:=  \kappa C \int_0^t \int_{\bbt^3}  |\nabla \bar \psi^{\kappa, \nu}_{\delta,\eps}|^2  dxd\tau ,\\
&J_2:=\int_0^t\int_{\bbt^3} \Big(C_1 \rho^{\kappa, \nu} + C_2 (-\rho s)(U^{\kappa, \nu}) \Big)\Big( \partial_\tau\bar \psi^{\kappa,\nu}_{\delta,\eps} + u^{\kappa,\nu}\cdot\nabla \bar \psi^{\kappa,\nu}_{\delta,\eps} \Big) dxd\tau .
\end{aligned}
\end{align*}
For $J_1$, using $\nabla \bar\psi^{\kappa,\nu}_{\delta,\eps} = \psi^{\kappa, \nu}_{\delta} * \nabla\eta_\eps$ by \eqref{psi-mol}, and $0\le  \psi^{\kappa, \nu}_{\delta}\le 1$ by \eqref{psi-bdd}, we have
\[
 \int_0^t \int_{\bbt^3}  |\nabla \bar\psi^{\kappa, \nu}_{\delta,\eps}|^2  dxd\tau \le C(\eps) \quad\mbox{independent of $\kappa$},
\]
from which,
\[
J_1\to 0 \quad \mbox{as } \kappa\to 0.
\]

For $J_2$, we will use the following uniform bound :
\beq\label{uni-vk}
\|u^{\kappa, \nu}\|_{L^2(0,T;H^1(\bbt^3))} \le C \quad\mbox{(independent of $\kappa,\delta,\eps$)}.
\eeq
This follows from  \eqref{pri-v} and 
\[
\|u^{\kappa, \nu} \|_{L^2(\bbt^3)}^2 \le C \Big[\|\nabla u^{\kappa, \nu} \|_{L^2(\bbt^3)} + \|\rho^{\kappa, \nu} \|_{L^1(\bbt^3)} \|\sqrt{\rho^{\kappa, \nu}} u^{\kappa, \nu}  \|_{L^2(\bbt^3)}^2 \Big] \le C,
\]
which is obtained by Lemma \ref{lem:fei} together with \eqref{pri-v}, \eqref{energy-d}, \eqref{pri-mass}.

\begin{lemma}\label{lem:fei} \cite[Lemma 3.2]{Feireisl_b}
Let $v\in W^{1,2}(\Omega)$ for a bounded domain $\Omega\subset\bbr^N$, and $\rho$ be a non-negative function such that
\[
0<M\le \int_\Omega \rho \,dx,\quad \int_\Omega \rho^\gamma dx \le E_0,
\] 
where $M, E_0$ and $\gamma>1$ are some constants.
Then there exists a constant $C=C(M,E_0)$ such that
\[
\|v\|_{L^2(\Omega)}^2 \le C\bigg( \|\nabla_x v\|_{L^2(\Omega)}^2 +\Big(\int_\Omega \rho |v| dx \Big) \bigg).
\]
\end{lemma}

By \eqref{uni-vk}, there exists $u^\nu\in L^2(0,T;H^1(\bbt^3))$ such that
\beq\label{wcu}
u^{\kappa, \nu}\weakto u^\nu \quad \mbox{weakly}~\mbox{in }~ L^2(0,T;H^1(\bbt^3)) \quad \mbox{as } \kappa\to 0 .
\eeq
Note that it follows from the uniform bound \eqref{psi-bdd}, there exists $\psi^{\nu}_{\delta}$  such that 
\beq\label{psinud}
0\le \psi^{\nu}_\delta \le 1
\eeq
and
\beq\label{psi-d}
\psi^{\kappa, \nu}_{\delta} \weakto \psi^{\nu}_\delta \quad \mbox{weakly-*} ~\mbox{in }~ L^\infty((0,T)\times\bbt^3) \quad \mbox{as } \kappa\to 0.
\eeq
To get the desired limit from $J_2$, we need to rewrite  $\partial_\tau\bar \psi^{\kappa,\nu}_{\delta,\eps} + u^{\kappa,\nu}\cdot\nabla \bar \psi^{\kappa,\nu}_{\delta,\eps}$ as follows:
First, take the mollification on the equation \eqref{psi-eq} with the mollifier $\eta_\eps$, to have
\beq\label{mopsi}
 \partial_t \bar\psi^{\kappa, \nu}_{\delta,\eps} + (\overline{\bar u_\delta^{\kappa, \nu}\cdot\nabla \psi_\delta^{\kappa, \nu}})_\eps = 0,
\eeq
from which, we have
\[
 \partial_t \bar\psi^{\kappa, \nu}_{\delta,\eps} + u^{\kappa, \nu} \cdot\nabla \bar\psi^{\kappa, \nu}_{\delta,\eps} = R^{\kappa, \nu}_{\delta,\eps} ,
\]
where
\[
R^{\kappa, \nu}_{\delta,\eps} := u^{\kappa, \nu} \cdot\nabla \bar\psi^{\kappa, \nu}_{\delta,\eps} 
-  (\overline{\bar u_\delta^{\kappa, \nu}\cdot\nabla \psi_\delta^{\kappa, \nu}})_\eps.
\]
equivalently,
\[
R^{\kappa, \nu}_{\delta,\eps} :=u^{\kappa, \nu} \cdot\nabla \bar\psi^{\kappa, \nu}_{\delta,\eps} 
 - \big(\overline{\div( \bar u^{\kappa, \nu}_\delta \psi^{\kappa, \nu}_{\delta} )}\big)_\eps + \big(\overline{(\div \bar u^{\kappa, \nu}_\delta) \psi^{\kappa, \nu}_{\delta}}\big)_\eps.
\]
By using the above representation, we rewrite $J_2$  as
\[
J_2:=\int_0^t\int_{\bbt^3} \Big(C_1 \rho^{\kappa, \nu} + C_2 (-\rho s)(U^{\kappa, \nu}) \Big) R^{\kappa, \nu}_{\delta,\eps} dxd\tau,
\]

We will first show that 
there exists $G^\nu \in L^\infty(0,T;L^2(\bbt^3))$ such that
\[
C_1 \rho^{\kappa, \nu} + C_2 (-\rho s)(U^{\kappa, \nu})  \to G^\nu\quad\mbox{strongly}~\mbox{in }~ L^2(0,T;H^{-1}(\bbt^3)) \quad \mbox{as } \kappa\to 0.
\]
For that, we may use the Aubin-Lions lemma (see \cite[Lemma 6.3]{Feireisl_b}):
\begin{lemma}\label{lem:lions}
Let $\{v_n\}_{n=1}^\infty$ be a sequence of functions such that $v_n$ are uniformly bounded in $L^2(0,T; L^q(\Omega))\cap L^\infty(0,T; L^1(\Omega))$ with $q>\frac{2N}{N+1}$. Furthermore, assume that
 $$
 \partial_t v_n\geq g_n,\qquad {\rm in }\quad \mathcal{D}^\prime((0,T)\times\Omega)
 $$
where $g_n$ are uniformly bounded in $L^1(0,T; W^{-m,r}(\Omega))$ with $m\geq1$ and $r>1$. Then $\{v_n\}_{n=1}^\infty$ contains a subsequence such that
$$
v_n\rightarrow v\qquad {\rm in }\quad L^2(0,T; H^{-1}(\Omega))\quad {\rm strongly}.
$$
\end{lemma}

In order to apply Lemma \ref{lem:lions} to the entropy inequality \eqref{ent-ineq}:
$$
(\rho^{\kappa,\nu} s^{\kappa,\nu})_t+{\rm div} (\rho^{\kappa,\nu} u^{\kappa,\nu}s^{\kappa,\nu})-\kappa{\rm div} \big(\frac{\nabla \theta^{\kappa,\nu}}{
\theta^{\kappa,\nu}}\big) +\kappa\frac{|\nabla\theta^{\kappa,\nu}|^2}{(\theta^{\kappa,\nu})^2}+\nu\frac{\mathbb{S}:\nabla u^{\kappa,\nu}}{\theta^{\kappa,\nu}}\geq0 \quad{\rm in} \quad \mathcal{D}^\prime((0,T)\times\bbt^3),
$$
we set
\[
(\rho^{\kappa,\nu} s^{\kappa,\nu})_t \ge -{\rm div} (\rho^{\kappa,\nu} u^{\kappa,\nu}s^{\kappa,\nu})+\kappa{\rm div} \big(\frac{\nabla \theta^{\kappa,\nu}}{
\theta^{\kappa,\nu}}\big) -\kappa\frac{|\nabla\theta^{\kappa,\nu}|^2}{(\theta^{\kappa,\nu})^2}-\nu\frac{\mathbb{S}:\nabla u^{\kappa,\nu}}{\theta^{\kappa,\nu}}:=\sum_{i=1}^4g_i^\kappa.
\]
Since it follows from \eqref{pri-ent} and \eqref{uni-vk} that $\rho^{\kappa,\nu} s^{\kappa,\nu}$ and $u^{\kappa,\nu}$ are respectively uniformly bounded in $L^2(0,T; L^2(\bbt^3))$  and  $L^2(0,T; L^6(\bbt^3))$, $g_1^\kappa$ is uniformly bounded in $L^2(0,T; W^{-1,\frac32}(\bbt^3))$.
Since $\sqrt{\kappa}\frac{\nabla \theta^{\kappa,\nu}}{
\theta^{\kappa,\nu}}$ is uniformly bounded in $L^2(0,T; L^2(\bbt^3))$ by \eqref{pri-v}, $g_2^\kappa$ is uniformly bounded in $L^2(0,T; W^{-1,2}(\bbt^3))$. Moreover, $g_3^\kappa+g^{\kappa}_4$ is uniformly bounded in $L^1(0,T; L^1(\bbt^3))\hookrightarrow L^1(0,T; W^{-1,\frac32}(\bbt^3))$.
Therefore, by the Aubin-Lions lemma, $\{\rho^{\kappa,\nu} s^{\kappa,\nu}\}_{\kappa>0}$ is pre-compact in $L^2(0,T; W^{-1,2}(\bbt^3))$.
Thus, there exists $G^\nu_*$ such that
\[
(-\rho s)(U^{\kappa, \nu}) \to G^\nu_* \quad\mbox{strongly}~\mbox{in }~ L^2(0,T;H^{-1}(\bbt^3)) \quad \mbox{as } \kappa\to 0,
\]
in addition, by \eqref{pri-ent},
\[
G^\nu_* \in L^\infty(0,T;L^2(\bbt^3)).
\]

Likewise, applying the Aubin-Lions lemma to the continuity equation:
\[
\partial_t \rho^{\kappa,\nu} + \div (\rho^{\kappa,\nu} u^{\kappa,\nu}) =0 ,\quad \mbox{in }~ \mathcal{D}^\prime((0,T)\times\bbt^3),
\]
and using \eqref{pri-mass} and \eqref{uni-vk}, we obtain that there exists $\rho^\nu \in L^\infty(0,T;L^2(\bbt^3))$ such that
\beq\label{rho-strong}
\rho^{\kappa, \nu} \to \rho^\nu\quad\mbox{strongly}~\mbox{in }~ L^2(0,T;H^{-1}(\bbt^3)) \quad \mbox{as } \kappa\to 0.
\eeq
Thus, putting $G^\nu:= C_1 \rho^\nu + C_2 G^\nu_*$, we have
\beq\label{Gb}
\|G^\nu\|_{L^\infty(0,T;L^2(\bbt^3))} \le C, \quad\mbox{ (independent of $\eps, \delta$) }
\eeq
and
\beq\label{strongg}
C_1 \rho^{\kappa, \nu} + C_2 (-\rho s)(U^{\kappa, \nu})  \to G^\nu\quad\mbox{strongly}~\mbox{in }~ L^2(0,T;H^{-1}(\bbt^3))  \quad \mbox{as } \kappa\to 0.
\eeq

Next, we will show that
\beq\label{claim-Rk}
R^{\kappa, \nu}_{\delta,\eps} \weakto R^{\nu}_{\delta,\eps} \quad\mbox{weakly} ~\mbox{in }~ L^2(0,T; H^1(\bbt^3)) \quad \mbox{as } \kappa\to 0,
\eeq
where
\beq\label{RNE}
R^{\nu}_{\delta,\eps} :=  u^{\nu} \cdot\nabla \bar\psi^{\nu}_{\delta,\eps} 
- \big(\overline{\div(  \bar u^{\nu}_\delta \psi^{\nu}_{\delta} )}\big)_\eps + \big(\overline{(\div \bar u^{\nu}_\delta) \psi^{\nu}_{\delta}}\big)_\eps.
\eeq
For that, we may derive a strong convergence of $\psi^{\kappa, \nu}_{\delta}$. 
Since the equation \eqref{psi-eq} can be rewritten as
\[
\partial_t  \psi^{\kappa, \nu}_{\delta} = - \div( \bar u^{\kappa, \nu}_\delta \psi^{\kappa, \nu}_{\delta} ) + (\div \bar u^{\kappa, \nu}_\delta) \psi^{\kappa, \nu}_{\delta},
\]
the uniform bounds \eqref{psi-bdd} and \eqref{uni-vk} imply that $\{\partial_t  \psi^{\kappa, \nu}_{\delta} \}_{\kappa>0}$ is bounded in $L^2(0,T;H^{-1}(\bbt^3))$. Then, this together with the bound \eqref{psi-bdd} and the Aubin-Lions Lemma implies 
\beq\label{st-psi}
\psi^{\kappa, \nu}_{\delta} \to \psi^{\nu}_{\delta}\quad\mbox{in }~ L^2(0,T;H^{-1}(\bbt^3)) \quad \mbox{as } \kappa\to 0 .
\eeq
Moreover, note that \eqref{uni-vk} with the mollification implies
\beq\label{highu}
\|\bar u^{\kappa, \nu}_\delta\|_{L^2(0,T;H^{2}(\bbt^3))} \le C \quad\mbox{ (independent of $\kappa$) } .
\eeq
Then, using the following Lemma \ref{lem:gencon} together with \eqref{uni-vk}, \eqref{psi-bdd},\eqref{st-psi} and \eqref{highu}, we have
\begin{align*}
\begin{aligned}
&\big(\overline{\div( \bar u^{\kappa, \nu}_\delta \psi^{\kappa, \nu}_{\delta} )}\big)_\eps  \weakto \big(\overline{\div( \bar u^{\nu}_\delta \psi^{\nu}_{\delta} )}\big)_\eps 
 \quad\mbox{weakly}\quad\mbox{in }~  L^2(0,T;H^1(\bbt^3)) ,\\
 &\big(\overline{(\div \bar u^{\kappa, \nu}_\delta) \psi^{\kappa, \nu}_{\delta}}\big)_\eps \weakto   \big(\overline{(\div \bar u^{\nu}_\delta) \psi^{\nu}_{\delta}}\big)_\eps 
 \quad\mbox{weakly}\quad\mbox{in }~  L^2(0,T;H^1(\bbt^3)) .
\end{aligned}
\end{align*}
On the other hand, since \eqref{psi-bdd} and the mollification yield
\[
\| \nabla \bar\psi^{\kappa, \nu}_{\delta,\eps} \|_{L^\infty(0,T;L^\infty(\bbt^3))} \le C \quad\mbox{ (independent of $\kappa$) }, 
\]
which together with \eqref{uni-vk} implies
\beq\label{uni-udp}
\| u^{\kappa,\nu} \cdot  \nabla \bar\psi^{\kappa, \nu}_{\delta,\eps} \|_{L^2(0,T;H^1(\bbt^3))} \le C \quad\mbox{ (independent of $\kappa$) } .
\eeq
Moreover, since \eqref{st-psi} with the mollification implies
\[
 \nabla \bar\psi^{\kappa, \nu}_{\delta,\eps} \to\nabla \bar\psi^{\nu}_{\delta,\eps} \quad\mbox{in }~ L^2(0,T;H^1(\bbt^3)) \quad \mbox{as } \kappa\to 0 ,
\]
this together with \eqref{wcu} and \eqref{uni-udp} implies
\[
 u^{\kappa,\nu} \cdot  \nabla \bar\psi^{\kappa, \nu}_{\delta,\eps} \weakto  u^{\nu} \cdot  \nabla \bar\psi^{\nu}_{\delta,\eps}   \quad\mbox{weakly}\quad\mbox{in }~ L^2(0,T;H^1(\bbt^3)) \quad \mbox{as } \kappa\to 0 .
\]
Therefore, we obtain the desired convergence \eqref{claim-Rk}.\\

\begin{lemma}\label{lem:gencon}
Assume that $\{u_n\}_{n\in\bbn}$ and $\{\psi_n\}_{n\in\bbn}$ are sequences such that
one of the following holds:
\begin{align*}
\begin{aligned} 
&{\rm (i) } \quad u_n \to u  \quad\mbox{in }~  L^2(0,T;H^1(\bbt^3))
 \quad{and}\quad \psi_n\weakto\psi   \quad\mbox{weakly}-*~\quad\mbox{in }~ L^\infty((0,T)\times\bbt^3)
; \\
& {\rm (ii) }\quad \psi_n \to \psi  \quad\mbox{in }~  L^2(0,T;H^{-1}(\bbt^3)) \quad{and}\\
&\qquad \|\psi_n\|_{ L^\infty((0,T)\times\bbt^3)} + \|u_n\|_{ L^2(0,T;H^1(\bbt^3))} +  \|\nabla{\rm div} u_n\|_{ L^2(0,T;L^2(\bbt^3))} \le C \, (\mbox{independent of } \, n).
\end{aligned}
\end{align*}
Then, for any fixed $\eps>0$, up to a subsequence, 
\begin{align*}
\begin{aligned}
&\big(\overline{\div( u_n  \psi_n )}\big)_\eps \weakto \big(\overline{\div( u  \psi )}\big)_\eps 
 \quad\mbox{weakly}\quad\mbox{in }~  L^2(0,T;H^1(\bbt^3)) \quad \mbox{as } n\to 0,\\
 &\big(\overline{(\div u_n) \psi_n }\big)_\eps \weakto   \big(\overline{(\div u) \psi }\big)_\eps 
 \quad\mbox{weakly}\quad\mbox{in }~  L^2(0,T;H^1(\bbt^3))  \quad \mbox{as } n\to 0.
\end{aligned}
\end{align*}
\end{lemma}
\begin{proof}
If the assumption (i) holds, then 
\[
 u_n  \psi_n \weakto u\psi \quad\mbox{and}\quad (\div u_n) \psi_n\weakto (\div u) \psi  \quad\mbox{weakly}\quad\mbox{in }~  L^2(0,T;L^2(\bbt^3)) .
\]  
Thus, thanks to the spatial mollification, we have the desired convergence.\\
If the assumption (ii)  holds, then up to a subsequence,
\[
 u_n  \psi_n \to u\psi \quad\mbox{and}\quad (\div u_n) \psi_n\to (\div u) \psi  \quad \mbox{in }~ \mathcal{D}^\prime((0,T)\times\bbt^3).
\]  
Moreover, since
\[
\|\big(\overline{\div( u_n  \psi_n )}\big)_\eps\|_{L^2(0,T;H^1(\bbt^3))} +\|\big(\overline{(\div u_n) \psi_n }\big)_\eps \|_{L^2(0,T;H^1(\bbt^3))} \le C\quad(\mbox{independent of } \,n),
\]
we have the desired convergence.
\end{proof}

Therefore, \eqref{strongg} and \eqref{claim-Rk} implies
\[
J_2\to \int_0^t \int_{\bbt^3} G^\nu R^{\nu}_{\eps,\delta}  dxd\tau,\quad\forall t\le T.
\]
Hence we have from \eqref{dwent} that 
\begin{align}
\begin{aligned}\label{final-k}
&\liminf_{\kappa\to0}\int_{\bbt^3}\Big(\mathcal{F}(\rho^{\kappa, \nu}) + \mathcal{G}(\mathcal{E}^{\kappa, \nu}-\bar{\mathcal{E}}) +\frac{|m^{\kappa,\nu}|^2}{2\rho^{\kappa,\nu}}  \Big)dx \\
&\quad\le C\eps + C\int_{\bbt^3} (-\rho s) (U_0|\bar U) dx +\int_0^t \int_{\bbt^3} G^\nu R^{\nu}_{\eps,\delta}  dxd\tau,\quad\forall t\le T.
\end{aligned}
\end{align}
Since it follows from \eqref{wcu} and \eqref{rho-strong}  that
\beq\label{bluecolor}
m^{\kappa, \nu}=\rho^{\kappa, \nu}u^{\kappa, \nu}\weakto \rho^{\nu}u^{\nu}:=m^\nu  \quad\mbox{weakly}~\mbox{in }~\mathcal{D}'((0,T)\times\bbt^3),
\eeq
using the weakly lower semi-continuity of the convex functionals $\rho\mapsto\mathcal{F}(\rho)$ and $(\rho,m)\mapsto\frac{ |m|^2}{\rho}$, we have
\beq\label{fmcon}
\int_{\bbt^3}\Big(\mathcal{F}(\rho^{\nu})  +\frac{|m^{\nu}|^2}{2\rho^{\nu}}  \Big)dx \le \liminf_{\kappa\to0}\int_{\bbt^3}\Big(\mathcal{F}(\rho^{\kappa, \nu})  +\frac{|m^{\kappa,\nu}|^2}{2\rho^{\kappa,\nu}}  \Big)dx.
\eeq
On the other hand, since the definition of the functional $\mathcal{G}$ yields that the quanity $f:=\mathcal{E}^{\kappa, \nu}-\bar{\mathcal{E}}$ satisfies 
\begin{align}
\begin{aligned}\label{L1}
\int_{\bbt^3}|f| dx&= \int_{|f|\le 1}|f| dx+  \int_{|f|> 1}|f| dx \le \sqrt{\int_{\bbt^3}dx}\sqrt{\int_{|f|\le 1}|f|^2 dx} +  \int_{|f|> 1}|f| dx \\
&\le \sqrt{\int_{\bbt^3} \mathcal{G}(f) dx }+ \int_{\bbt^3}\mathcal{G}(f) dx ,
\end{aligned}
\end{align}
and  it follows from \eqref{final-k} that $\int_{\bbt^3}\mathcal{G}(f) dx$ is  uniformly bounded in $\kappa$, we find that there exists $\mathcal{E}^\nu \in L^\infty(0,T;\mathcal{M}(\bbt^3))$ such that
\[
\mathcal{E}^{\kappa, \nu}\weakto\mathcal{E}^\nu   \quad\mbox{weakly}-*~\quad\mbox{in }~  L^\infty(0,T;\mathcal{M}(\bbt^3)).
\]
Therefore, this and \eqref{L1} imply
\[
\|\mathcal{E}^\nu-\bar{\mathcal{E}}\|_{L^\infty(0,T;\mathcal{M}(\bbt^3))} \le\liminf_{\kappa\to0} \sup_{t\in[0,T]} \left( \sqrt{\int_{\bbt^3} \mathcal{G}(\mathcal{E}^{\kappa, \nu}-\bar{\mathcal{E}}) dx } +\int_{\bbt^3} \mathcal{G}(\mathcal{E}^{\kappa, \nu}-\bar{\mathcal{E}}) dx \right) .
\]
Hence, we obtain from \eqref{final-k} and \eqref{fmcon} that
\beq\label{sim-ineq}
 \sup_{t\in[0,T]} \int_{\bbt^3}\Big(\mathcal{F}(\rho^{\nu})  +\frac{|m^{\nu}|^2}{2\rho^{\nu}} \Big)dx + \|\mathcal{E}^\nu-\bar{\mathcal{E}}\|_{L^\infty(0,T;\mathcal{M}(\bbt^3))} \le C
 \sup_{t\in[0,T]} \Big(\sqrt{ \mathcal{A}^{\nu}_{\eps,\delta}} + \mathcal{A}^{\nu}_{\eps,\delta} \Big),
\eeq
where
\[
\mathcal{A}^{\nu}_{\eps,\delta} :=C\eps + \int_{\bbt^3} (-\rho s) (U_0|\bar U) dx + \left|\int_0^t \int_{\bbt^3} G^\nu R^{\nu}_{\eps,\delta}  dxd\tau \right|.
\]
Also, note that the convergences \eqref{bluecolor} and \eqref{rho-strong} imply
\beq\label{mass-lim}
\partial_t \rho^\nu+ \div (m^\nu) =0\quad \mbox{in }~ \mathcal{D}^\prime((0,T)\times\bbt^3).
\eeq

\subsection{Passing to the limit as $\delta\to 0$}
Since only the term $R^{\nu}_{\eps,\delta}$ in \eqref{sim-ineq} depends on $\delta$, we will pass to the limit on $R^{\nu}_{\eps,\delta}$ as $\delta\to 0$ as $\delta\to 0$.\\

We first recall \eqref{RNE} as
\[
R^{\nu}_{\delta,\eps} := u^{\nu} \cdot\nabla \bar\psi^{\nu}_{\delta,\eps}  - \big(\overline{\div(  \bar u^{\nu}_\delta \psi^{\nu}_{\delta})}\big)_\eps + \big(\overline{(\div \bar u^{\nu}_\delta) \psi^{\nu}_{\delta}}\big)_\eps  .
\]
By \eqref{psinud}, up to a subsequence, there exists $\psi^{\nu}$ such that $0\le \psi^{\nu} \le 1$ and
\beq\label{last-psi}
\psi^{\nu}_{\delta} \weakto \psi^{\nu}\quad\mbox{weakly-*} \quad\mbox{in }~ L^\infty((0,T)\times\bbt^3) \quad \mbox{as } \delta\to 0,
\eeq
which yields that (by the mollification)
\[
\nabla\bar\psi^{\nu}_{\delta,\eps}  \weakto \nabla \bar\psi^{\nu}_{\eps} \quad\mbox{weakly-*}  \quad\mbox{in }~ L^\infty(0,T; H^1(\bbt^3)) \quad \mbox{as } \delta\to 0 .
\]
This together with $u^\nu\in L^2(0,T;H^1(\bbt^3))$ implies
\[
u^{\nu} \cdot\nabla \bar\psi^{\nu}_{\delta,\eps} \weakto u^{\nu} \cdot\nabla \bar\psi^{\nu}_{\eps}\quad\mbox{weakly}  \quad\mbox{in }~  L^2(0,T;H^1(\bbt^3)) \quad \mbox{as } \delta \to 0.
\]
Since $u^\nu\in L^2(0,T; H^1(\bbt^3))$ implies
\[
\bar u_\delta^\nu\to u^\nu \quad\mbox{in }~ L^2(0,T;H^{1}(\bbt^3)) \quad\mbox{strongly}  ~\mbox{as } \delta\to 0,
\]
which together with  Lemma \ref{lem:gencon} and \eqref{last-psi} yields
\begin{align*}
\begin{aligned}
&\big(\overline{\div(  \bar u^{\nu}_\delta \psi^{\nu}_{\delta})}\big)_\eps  \weakto  \big(\overline{\div\big(u^{\nu} \psi^{\nu} \big)} \big)_\eps 
 \quad\mbox{weakly}\quad\mbox{in }~  L^2(0,T;H^1(\bbt^3)) ,\\
 &\big(\overline{(\div \bar u^{\nu}_\delta) \psi^{\nu}_{\delta}}\big)_\eps \weakto
  \big(\overline{(\div u^{\nu})\psi^{\nu}}\big)_\eps 
 \quad\mbox{weakly}\quad\mbox{in }~  L^2(0,T;H^1(\bbt^3)) .
\end{aligned}
\end{align*}

Hence, we have
\beq\label{int-Rgo}
R^{\nu}_{\eps,\delta} \weakto R^{\nu}_{\eps} \quad\mbox{weakly} \quad\mbox{in }~ L^2(0,T;H^1(\bbt^3)),
\eeq
where 
\[
R^{\nu}_{\eps}:=u^{\nu} \cdot\nabla \bar\psi^{\nu}_{\eps} -  \big(\overline{\div\big(u^{\nu} \psi^{\nu} \big)} \big)_\eps+ \big(\overline{(\div u^{\nu})\psi^{\nu}}\big)_\eps.
\]
This and \eqref{Gb} with \eqref{sim-ineq} yield
\beq\label{sim-ineq2}
 \sup_{t\in[0,T]} \int_{\bbt^3}\Big(\mathcal{F}(\rho^{\nu})  +\frac{|m^{\nu}|^2}{2\rho^{\nu}} \Big)dx + \|\mathcal{E}^\nu-\bar{\mathcal{E}}\|_{L^\infty(0,T;\mathcal{M}(\bbt^3))} \le C \sup_{t\in[0,T]}  \Big(\sqrt{ \mathcal{A}^{\nu}_{\eps}} + \mathcal{A}^{\nu}_{\eps} \Big),
\eeq
where
\[
\mathcal{A}^{\nu}_{\eps} :=C\eps + \int_{\bbt^3} (-\rho s) (U_0|\bar U) dx +\left| \int_0^t \int_{\bbt^3} G^\nu R^{\nu}_{\eps}  dxd\tau\right| .
\]

\subsection{Passing to the limit as $\eps\to 0$}
We first rewrite $R^{\nu}_{\eps}$ as
\begin{align*}
\begin{aligned}
R^{\nu}_{\eps}&=\div\big(u^{\nu} \bar\psi^{\nu}_\eps \big) - (\div u^\nu) \bar\psi^{\nu}_\eps -  \big(\overline{\div\big(u^{\nu} \psi^{\nu} \big)} \big)_\eps+ \big(\overline{(\div u^{\nu})\psi^{\nu}}\big)_\eps\\
&=\underbrace{\Big(\big(\overline{(\div u^{\nu})\psi^{\nu}}\big)_\eps- (\div u^\nu) \bar\psi^{\nu}_\eps\Big)}_{=: K_1} +\underbrace{ \Big(\div\big(u^{\nu} \bar\psi^{\nu}_\eps \big)-  \big(\overline{\div\big(u^{\nu} \psi^{\nu} \big)}\big)_\eps \Big) }_{=: K_2} .
\end{aligned}
\end{align*}
First of all, note that
\beq\label{uni-k1}
\|K_1\|_{L^2((0,T)\times\bbt^3)}\le C \quad (\mbox{independent of } \eps).
\eeq
Since $\psi^\nu\in L^\infty((0,T)\times\bbt^3)\subset L^q((0,T)\times\bbt^3)$ for $q<\infty$,
\[
\bar \psi^{\nu}_\eps \to \psi^{\nu}\quad \mbox{strongly}  \quad\mbox{in }~ L^q((0,T)\times\bbt^3) \quad \mbox{as } \eps\to 0,
\]
which together with $u^\nu\in L^2(0,T;H^1(\bbt^3))$ implies
\[
 (\div u^\nu) \bar \psi^{\nu}_\eps \to  (\div u^\nu) \psi^{\nu}\quad\mbox{in }~ L^{r_0}((0,T)\times\bbt^3)\quad\mbox{for some } r_0 \in (1,2).
\]
Moreover, using $ (\div u^\nu) \psi^{\nu}\in L^2((0,T)\times\bbt^3)$, we have
\[
 \big(\overline{(\div u^{\nu})\psi^{\nu}}\big)_\eps\to  (\div u^\nu) \psi^{\nu}\quad\mbox{in }~ L^2((0,T)\times\bbt^3).
\]
Thus, 
\[
K_1 \to 0\quad\mbox{in }~ L^{r_0}((0,T)\times\bbt^3),
\]
which together with \eqref{uni-k1} yields
\[
K_1\weakto 0 \quad\mbox{in }~ L^2((0,T)\times\bbt^3).
\]

For $K_2$, we use the following lemma on the Lions commutator estimate:

\begin{lemma} \cite[Lemma 2.3]{Lions_book1} \label{lem:comm}
 There exists a constant $C$ such that for any $\eps>0$, any functions $f\in W^{1,\alpha}(\bbt^3)$ and $g\in L^{\beta}(\bbt^3)$ with $0\le \frac{1}{\alpha}+\frac{1}{\beta}=\frac{1}{r}\le1$ satisfy
\[
\| \big( \overline{\div\big(fg \big)}\big)_\eps -\div\big(f \bar g_\eps \big) \|_{L^r(\bbt^3)} \le
C  \|f\|_{W^{1,\alpha}(\bbt^3)} \|g\|_{L^{\beta}(\bbt^3)}.
\]
In addition, if $r<\infty$, then
\[
\| \big( \overline{\div\big(fg \big)}\big)_\eps -\div\big(f \bar g_\eps \big) \|_{L^r(\bbt^3)} \to 0 \quad \mbox{as } \eps\to 0.
\]
\end{lemma}

We may apply Lemma \ref{lem:comm} to our case: $ f=u^\nu, g=\psi^\nu, \alpha=2, \beta=\infty$ and $r=2$. Since $ \|u^\nu\|_{L^2(0,T;H^{1}(\bbt^3))}$ and $\|\psi^\nu\|_{L^\infty((0,T)\times\bbt^3) }$ are uniformly bounded in $\nu$, Lemma \ref{lem:comm} implies that
\beq\label{k11con}
\|K_2\|_{L^2(0,T;L^2(\bbt^3))} \le C \|u^\nu\|_{L^2(0,T;H^{1}(\bbt^3))} \|\psi^\nu\|_{L^\infty((0,T)\times\bbt^3)} \le C \quad \mbox{(independent of $\nu$)},
\eeq
and
\beq\label{k12con}
\|K_2(\cdot, t)\|_{L^2(\bbt^3))} \to 0 \quad \mbox{as } \eps\to 0, \quad \mbox{for a.e.}~ t\in(0,T),
\eeq
where we used the fact that $u^\nu(\cdot, t) \in H^{1}(\bbt^3)$ and $\psi^\nu(\cdot, t) \in L^{2}(\bbt^3)$ for a.e. $t\in(0,T)$, and uniformly in $\nu$.\\
Thus, it follows from  \eqref{k11con} and \eqref{k12con} that for $q<2$,
\[
\|K_2(\cdot, t)\|_{L^2(\bbt^3)} \to 0 \quad\mbox{in }~ L^q((0,T))\quad \mbox{as } \eps\to 0,
\]
which together with \eqref{k11con} yields
\[
K_2 \weakto 0 \quad\mbox{in }~ L^2((0,T)\times\bbt^3)\quad \mbox{as } \eps\to 0.
\]
Therefore,
\beq\label{rde}
 R^{\nu}_{\eps} \weakto 0 \quad\mbox{in }~ L^2((0,T)\times\bbt^3)\quad \mbox{as } \eps\to 0.
\eeq
which together with \eqref{Gb} yields
\[
\int_0^t \int_{\bbt^3} G^\nu R^{\nu}_{\eps}  dxd\tau \to 0 \quad \mbox{as }  \eps\to 0,\quad\forall t\in[0,T].
\]
Hence, taking $\eps\to 0$ in \eqref{sim-ineq2}, we have
\begin{align}
\begin{aligned}\label{semi-fine}
&\sup_{[0,T]}\int_{\bbt^3}\Big(\mathcal{F}(\rho^{\nu})   +\frac{|m^{\nu}|^2}{2\rho^{\nu}}\Big)dx + \|\mathcal{E}^\nu-\bar{\mathcal{E}}\|_{L^\infty(0,T;\mathcal{M}(\bbt^3))} \\
&\qquad \le C \sqrt{\int_{\bbt^3} (-\rho s) (U_0|\bar U) dx}+ C\int_{\bbt^3} (-\rho s) (U_0|\bar U) dx.
\end{aligned}
\end{align}

\subsection{Passing to the limit as $\nu\to 0$} \label{subsec:nu}
First of all, by \eqref{pri-mass}, there exists $\rho \in L^\infty(0,T;L^2(\bbt^3))$ such that
\beq\label{fin-lim-rho}
\rho^{\nu} \weakto \rho \quad\mbox{weakly-*}~\mbox{in }~L^\infty(0,T;L^2(\bbt^3)).
\eeq
Since the H\"older inequality with $|m^\nu| =\sqrt{\rho^\nu} \frac{|m^{\nu}|}{\sqrt{\rho^\nu}}$ yields
\[
\|m^\nu\|_{ L^\infty(0,T;L^{\frac{4}{3}}(\bbt^3))} \le \|\sqrt{\rho^\nu}\|_{L^\infty(0,T;L^4(\bbt^3))} \Big\| \frac{m^\nu}{\sqrt{\rho^\nu}}\Big\|_{L^\infty(0,T;L^2(\bbt^3))},
\]
the uniform bounds \eqref{pri-mass} and \eqref{energy-d} imply the uniform bound of $m^\nu$ in $ L^\infty(0,T;L^{\frac{4}{3}}(\bbt^3))$. 
Thus, there exists $m \in L^\infty(0,T;L^{\frac{4}{3}}(\bbt^3))$ such that
\beq\label{fin-lim-m}
m^{\nu} \weakto m \quad\mbox{weakly-*}~\mbox{in }~L^\infty(0,T;L^{\frac{4}{3}}(\bbt^3)).
\eeq
Note that it follows from \eqref{semi-fine} that there exists $\mathcal{E} \in L^\infty(0,T;\mathcal{M}(\bbt^3))$ such that
\beq\label{fin-lim-E}
\mathcal{E}^{\nu} \weakto \mathcal{E} \quad\mbox{weakly-*}~\mbox{in }~L^\infty(0,T;\mathcal{M}(\bbt^3)).
\eeq
Therefore, \eqref{fin-lim-E} and the weakly lower semi-continuity of the convex functionals in \eqref{semi-fine} together with \eqref{fin-lim-rho} and  \eqref{fin-lim-m}, we have
\begin{align}
\begin{aligned}\label{1sts}
&\sup_{[0,T]}\int_{\bbt^3}\Big(\mathcal{F}(\rho)  +\frac{|m|^2}{2\rho}\Big)dx + \|\mathcal{E}-\bar{\mathcal{E}}\|_{L^\infty(0,T;\mathcal{M}(\bbt^3))} \\
&\qquad \le C \sqrt{\int_{\bbt^3} (-\rho s) (U_0|\bar U) dx}+ C\int_{\bbt^3} (-\rho s) (U_0|\bar U) dx,
\end{aligned}
\end{align}
which gives \eqref{thmine}.\\
We also obtain from \eqref{mass-lim} that
\beq\label{fincon}
\partial_t \rho+ \div (m) =0\quad \mbox{in }~ \mathcal{D}^\prime((0,T)\times\bbt^3).
\eeq

\subsection{Uniqueness}
We here prove the last part  of Theorem \ref{thm:main}, for the uniqueness.\\
Let  $(U_0^n)_{n\in\bbn}$ be a sequence of initial data such that
\beq\label{inia}
\int_{\bbt^3} (-\rho s) (U_0^n|\bar U) dx \to 0 \quad\mbox{as } n\to \infty.
\eeq
Then, by \eqref{1sts},
\[
\liminf_{n\to \infty}\left[\sup_{[0,T]}\int_{\bbt^3}\Big(\mathcal{F}(\rho^n)  +\frac{|m^{n}|^2}{2\rho^n} \Big)dx + \|\mathcal{E}^n-\bar{\mathcal{E}}\|_{L^\infty(0,T;\mathcal{M}(\bbt^3))} \right]=0.
\]
Thus, as $n\to\infty$,
\[
{\mathcal{E}}^n \to \bar{\mathcal{E}}  \quad\mbox{in }~ L^\infty(0,T;\mathcal{M}(\bbt^3)) ,
\]
and
\[
\frac{|m^{n}|}{\sqrt{\rho^n}}  \to 0 \quad\mbox{in }~L^\infty(0,T;L^2(\bbt^3)) .
\]
Since
\begin{align*}
\begin{aligned}
\rho^n &= \rho^n \mathbf{1}_{\rho^n\in[\rho_*, \rho^*]} + \rho^n \mathbf{1}_{\rho^n\in(0,\rho_*)}+ \rho^n \mathbf{1}_{\rho^n\in(\rho^*,\infty)} \\
&\le \rho^* + | \rho^n \mathbf{1}_{\rho^n\in(0,\rho_*)}-\rho_*| + \rho_* + | \rho^n \mathbf{1}_{\rho^n\in(\rho^*,\infty)}-\rho^*| + \rho^*,
\end{aligned}
\end{align*}
where $\rho_*$ and $\rho^*$ are defined in \eqref{def-FG0}, it follows from \eqref{def-FG0} and $\liminf_{n \to \infty} \sup_{[0,T]}\int_{\bbt^3} \mathcal{F}(\rho^\eps) dx=0$ that
\[
\|\rho^n\|_{L^\infty(0,T;L^2(\bbt^3))}^2 \le C+C \sup_{[0,T]}\int_{\bbt^3} \mathcal{F}(\rho^n) dx \le C,
\]
which implies that there exists $\rho$ such that
\beq\label{rhoef}
\rho^n \weakto \rho \quad\mbox{weakly-*} \quad\mbox{in }~L^\infty(0,T;L^2(\bbt^3)).
\eeq
Therefore,
\[
\|m^n\|_{ L^\infty(0,T;L^{\frac{4}{3}}(\bbt^3))} \le \|\sqrt{\rho^n}\|_{L^\infty(0,T;L^4(\bbt^3))} \Big\| \frac{m^{n}}{\sqrt{\rho^n}}\Big\|_{L^\infty(0,T;L^2(\bbt^3))} \le C\Big\| \frac{m^{n}}{\sqrt{\rho^n}}\Big\|_{L^\infty(0,T;L^2(\bbt^3))} \to 0.
\]
Thus, it remains to prove that $\rho=\bar\rho$.\\
To this end, we will use the equation \eqref{fincon}: 
\beq\label{limconeq}
\partial_t \rho^n+ \div_x (m^n) =0 \quad \mbox{in }~ \mathcal{D}^\prime((0,T)\times\bbt^3).
\eeq
 For test functions for the above equation, we recall the same functions as in \eqref{time-mol}:
for any $t\in(0,T)$, and any $\zeta<t/2$, we consider
\[
\varphi_{t, \zeta} (s) := \int_0^s \Big(\phi_\zeta (z) -\phi_\zeta (z- t)  \Big) dz,
\]
where
\[
\phi_\zeta(s):=\frac{1}{\zeta}\phi \big(\frac{s-\zeta}{\zeta}\big)\quad\mbox{for any } \zeta>0 ,
\]
and $\phi:\bbr\to\bbr$ is a non-negative smooth function such that $\phi (s)=\phi(-s)$, $\int_{\bbr}\phi=1$ and $\mbox{supp } \phi = [-1,1]$.\\
Then we consider the following test functions for the continuity equation:
\[
\Phi(x) \varphi_{t, \zeta} (s),
\]
where $\Phi\in C^\infty(\bbt^3)$. That is,
\[
\int_0^\infty \int_{\bbt^3} \Big(   \Phi(x) (-\varphi_{t, \zeta}' (s)) \rho^n  \Big)dxds =\int_0^\infty \int_{\bbt^3} \Big( \varphi_{t, \zeta} (s)\nabla \Phi \cdot m^n \Big)dxds.
\]
Thus,
\[
\int_0^\infty \int_{\bbt^3} \Big(   \Phi(x) \phi_\zeta (s- t)  \rho^n  \Big)dxds =\int_0^\infty \int_{\bbt^3} \Big(   \Phi(x) \phi_\zeta (s)  \rho^n  \Big)dxds + \int_0^\infty \int_{\bbt^3} \Big( \varphi_{t, \zeta} (s)\nabla \Phi \cdot m^n \Big)dxds .
\]
Note that since $\rho^n\in L^\infty(0,T;L^2(\bbt^3))$ and $m^n \in L^2(0,T;L^{\frac{4}{3}}(\bbt^3))$ by \eqref{thmlim}, it follows from \eqref{limconeq} and Aubin-Lions Lemma that $\rho^n\in C([0,T];L^1(\bbt^3))$.\\
Therefore, taking $\zeta\to0$ above, we find that 
\[
 \int_{\bbt^3}    \Phi(x)  \rho^n(x,t)  dx =\int_{\bbt^3}  \Phi(x)  \rho^n_0(x)  dx + \int_0^t \int_{\bbt^3} \nabla \Phi \cdot m^n dxds .
\]
Note that \eqref{inia} yields
\beq\label{rhof}
\int_{\bbt^3} |\rho_0^n-\bar\rho|^2 dx \le \int_{\bbt^3}  p_e(\rho_0^n|\bar\rho) dx \to 0 \quad\mbox{as } n\to \infty.
\eeq
Since
\[
\left|\int_0^t \int_{\bbt^3} \nabla \Phi \cdot m^n dxds \right| \le \|\nabla \Phi \|_{L^{4}(\bbt^3)} \|m^n\|_{ L^\infty(0,T;L^{\frac{4}{3}}(\bbt^3))} \to 0,
\]
taking $n\to\infty$ and using \eqref{rhoef} and \eqref{rhof}, 
we have
\[
 \int_{\bbt^3}    \Phi(x)  \rho (x,t)  dx =\int_{\bbt^3}  \Phi(x)  \bar\rho(x)  dx ,\quad\forall t\in [0,T].
\]
Hence, $\rho=\bar\rho$ a.e. on $\bbt^3\times[0,T]$.\\

\noindent{\bf Conflict of Interest:} The authors declared that they have no conflicts of interest to this work.

\bibliography{Kang-Vasseur2019}

\begin{thebibliography}{10}

\bibitem{AKV}
S.~Akopian, M.-J. Kang, and A.~Vasseur.
\newblock {Inviscid limit to the shock waves for the fractal Burgers equation}.
\newblock {\em Commun. Math. Sci.}, 18:1477--1491, 2020.

\bibitem{AKKMM}
H.~Al~Baba, C.~Klingenberg, O.~Kreml, V.~M\'acha, and S.~Markfelder.
\newblock {Non-uniqueness of admissible weak solution to the Riemann problem
  for the full Euler system in two dimensions}.
\newblock {\em SIAM J. Math. Anal.}, 52(2):1729--1760, 2020.

\bibitem{BKM}
J.~B$\check{\rm r}$ezina, O.~Kreml, and V.~M\'acha.
\newblock {Non-uniqueness of delta shocks and contact discontinuities in the
  multi-dimensional model of Chaplygin gas}.
\newblock {\em Nonlinear Differ. Equ. Appl.}, 28:Article number: 13, 2021.

\bibitem{BB}
S.~Bianchini and A.~Bressan.
\newblock {Vanishing viscosity solutions to nolinear hyperbolic systems}.
\newblock {\em Ann. of Math.}, 166:223--342, 2005.

\bibitem{ChenChen}
G.-Q. Chen and J.~Chen.
\newblock Stability of rarefaction waves and vacuum states for the
  multidimensional euler equations,.
\newblock {\em J. Hyperbolic Differ. Equ.}, 4:105--122, 2007.

\bibitem{Chen1}
G.-Q. Chen, H.~Frid, and Y.~Li.
\newblock Uniqueness and stability of {R}iemann solutions with large
  oscillation in gas dynamics.
\newblock {\em Comm. Math. Phys.}, 228(2):201--217, 2002.

\bibitem{Ch}
E.~Chiodaroli.
\newblock {A counterexample to well-posedness of entropy solutions to the
  compressible Euler system}.
\newblock {\em J. Hyperbolic Differ. Equ.}, 11(3):493--519, 2014.

\bibitem{CDK}
E.~Chiodaroli, C.~De~Lellis, and O.~Kreml.
\newblock {Global ill-posedness of the isentropic system of gas dynamics}.
\newblock {\em Comm. Pure Appl. Math.}, 68(7):1157--1190, 2015.

\bibitem{CFK}
E.~Chiodaroli, E.~Feireisl, and O.~Kreml.
\newblock {On the weak solutions to the equations of a compressible heat
  conducting gas}.
\newblock {\em Ann. Inst. H. Poincar\'e Anal. Non Linaire}, 32(1):225--243,
  2015.

\bibitem{CK}
E.~Chiodaroli and O.~Kreml.
\newblock {On the energy dissipation rate of solutions to the compressible
  isentropic Euler system}.
\newblock {\em Arch. Ration. Mech. Anal.}, 214(3):1019--1049, 2014.

\bibitem{CKKV}
K.~Choi, M.-J. Kang, Y.~Kwon, and A.~Vasseur.
\newblock Contraction for large perturbations of traveling waves in a
  hyperbolic-parabolic system arising from a chemotaxis model.
\newblock {\em Math. Models Methods Appl. Sci.}, 30:387--437, 2020.

\bibitem{CV}
K.~Choi and A.~Vasseur.
\newblock Short-time stability of scalar viscous shocks in the inviscid limit
  by the relative entropy method.
\newblock {\em SIAM J. Math. Anal.}, 47:1405--1418, 2015.

\bibitem{Dafermos2}
C.~Dafermos.
\newblock {\em Hyperbolic conservation laws in continuum physics}, volume 325
  of {\em Grundlehren der Mathematischen Wissenschaften [Fundamental Principles
  of Mathematical Sciences]}.
\newblock Springer-Verlag, Berlin, 2000.

\bibitem{Dafermos1}
C.~M. Dafermos.
\newblock The second law of thermodynamics and stability.
\newblock {\em Arch. Rational Mech. Anal.}, 70(2):167--179, 1979.

\bibitem{DS09}
C.~De~Lellis and L.~Sz\'ekelyhidi.
\newblock {The Euler equations as a differential inclusion}.
\newblock {\em Ann. of Math.}, 170(3):1417--1436, 2009.

\bibitem{DS10}
C.~De~Lellis and L.~Sz\'ekelyhidi.
\newblock {On admissibility criteria for weak solutions of the Euler
  equations}.
\newblock {\em Arch. Ration. Mech. Anal.}, 195(1):225--260, 2010.

\bibitem{DiPerna}
R.~J. DiPerna.
\newblock Uniqueness of solutions to hyperbolic conservation laws.
\newblock {\em Indiana Univ. Math. J.}, 28(1):137--188, 1979.

\bibitem{Feireisl_b}
E.~Feireisl.
\newblock {Dynamics of Viscous Compressible Fluids, Oxford Science Publication,
  Oxford}.
\newblock 2004.

\bibitem{FKKM}
E.~Feireisl, C.~Klingenberg, O.~Kreml, and S.~Markfelder.
\newblock {On oscillatory solutions to the complete Euler system}.
\newblock {\em J. Differential Equations}, 269(2):1521--1543, 2020.

\bibitem{FK15}
E.~Feireisl and O.~Kreml.
\newblock {Uniqueness of rarefaction waves in multidimensional compressible
  Euler system}.
\newblock {\em J. Hyperbolic Differ. Equ.}, 12:489--499, 2015.

\bibitem{FKV15}
E.~Feireisl, O.~Kreml, and A.~Vasseur.
\newblock {Stability of the isentropic Riemann solutions of the full
  multidimensional Euler system}.
\newblock {\em SIAM J. Math. Anal.}, 47:2416--2425, 2015.

\bibitem{HWWY}
F.~M. Huang, Y.~Wang, Y.~Wang, and T.~Yang.
\newblock The limit of the {B}oltzmann equation to the {E}uler equations.
\newblock {\em SIAM J. Math. Anal.}, 45:1741--1811, 2013.

\bibitem{HWY1}
F.~M. Huang, Y.~Wang, and T.~Yang.
\newblock Fluid dynamic limit to the {R}iemann solutions of {E}uler equations:
  {I}. {S}uperposition of rarefaction waves and contact discontinuity.
\newblock {\em Kinetic and Related Models}, 3:685--728, 2010.

\bibitem{HWY}
F.~M. Huang, Y.~Wang, and T.~Yang.
\newblock Hydrodynamic limit of the {B}oltzmann equation with contact
  discontinuities.
\newblock {\em Comm. Math. Phys.}, 295:293--326, 2010.

\bibitem{HWY2}
F.~M. Huang, Y.~Wang, and T.~Yang.
\newblock Vanishing viscosity limit of the compressible {N}avier-{S}tokes
  equations for solutions to {R}iemann problem.
\newblock {\em Arch. Ration. Mech. Anal.}, 203:379--413, 2012.

\bibitem{Kang}
M.-J. Kang.
\newblock Non-contraction of intermediate admissible discontinuities for {3-D}
  planar isentropic magnetohydrodynamics.
\newblock {\em Kinet. Relat. Models}, 11(1):107--118, 2018.

\bibitem{Kang19}
M.-J. Kang.
\newblock {$L^2$}-type contraction for shocks of scalar viscous conservation
  laws with strictly convex flux.
\newblock {\em J. Math. Pures Appl.}, 145:1--43, 2021.

\bibitem{Kang-V-2}
M.-J. Kang and A.~Vasseur.
\newblock Asymptotic analysis of {V}lasov-type equations under strong local
  alignment regime.
\newblock {\em Math. Mod. Meth. Appl. Sci.}, 25(11):2153--2173, 2015.

\bibitem{KVARMA}
M.-J. Kang and A.~Vasseur.
\newblock Criteria on contractions for entropic discontinuities of systems of
  conservation laws.
\newblock {\em Arch. Ration. Mech. Anal.}, 222(1):343--391, 2016.

\bibitem{Kang-V-1}
M.-J. Kang and A.~Vasseur.
\newblock {$L^2$}-contraction for shock waves of scalar viscous conservation
  laws.
\newblock {\em Annales de l'Institut Henri Poincar\'e (C) : Analyse non
  lin\'eaire}, 34(1):139Ð156, 2017.

\bibitem{Kang-V-NS17}
M.-J. Kang and A.~Vasseur.
\newblock {Contraction property for large perturbations of shocks of the
  barotropic {Navier-Stokes} system}.
\newblock {\em J. Eur. Math. Soc. (JEMS)}, 23:585--638, 2021.

\bibitem{KV-unique19}
M.-J. Kang and A.~Vasseur.
\newblock {Uniqueness and stability of entropy shocks to the isentropic Euler
  system in a class of inviscid limits from a large family of Navier-Stokes
  systems}.
\newblock {\em Invent. Math}, 224(1):55--146, 2021.

\bibitem{KVW}
M.-J. Kang, A.~Vasseur, and Y.~Wang.
\newblock {$L^2$}-contraction for planar shock waves of multi-dimensional
  scalar viscous conservation laws.
\newblock {\em J. Differential Equations}, 267(5):2737--2791, 2019.

\bibitem{KKMM}
C.~Klingenberg, O.~Kreml, V.~M\'acha, and S.~Markfelder.
\newblock {Shocks make the Riemann problem for the full Euler system in
  multiple space dimensions ill-posed}.
\newblock {\em Nonlinearity}, 33(12):6517--6540, 2020.

\bibitem{KM18}
C.~Klingenberg and S.~Markfelder.
\newblock {The Riemann problem for the multidimensional isentropic system of
  gas dynamics is ill-posed if it contains a shock}.
\newblock {\em Arch. Rational Mech. Anal.}, 227:967--994, 2018.

\bibitem{Krupa1}
S.~Krupa.
\newblock {Criteria for the a-contraction and stability for the
  piecewise-smooth solutions to hyperbolic balance laws}.
\newblock {\em Commun. Math. Sci.}, 18(6):1493--1537, 2020.

\bibitem{Krupa2}
S.~Krupa.
\newblock {Finite time stability for the Riemann problem with extremal shocks
  for a large class of hyperbolic systems}.
\newblock {\em J. Differential Equations}, 273:122--171, 2021.

\bibitem{Leger}
N.~Leger.
\newblock {$L^2$} stability estimates for shock solutions of scalar
  conservation laws using the relative entropy method.
\newblock {\em Arch. Ration. Mech. Anal.}, 199(3):761--778, 2011.

\bibitem{LV}
N.~Leger and A.~Vasseur.
\newblock Relative entropy and the stability of shocks and contact
  discontinuities for systems of conservation laws with non-{BV} perturbations.
\newblock {\em Arch. Ration. Mech. Anal.}, 201(1):271--302, 2011.

\bibitem{LWW2}
L.~A. Li, D.~H. Wang, and Y.~Wang.
\newblock Vanishing dissipation limit to the planar rarefaction wave for the
  three-dimensional compressible {N}avier-{S}tokes-{F}ourier equations.
\newblock {\em arXiv:2101.04291}.

\bibitem{LWW1}
L.~A. Li, D.~H. Wang, and Y.~Wang.
\newblock Vanishing viscosity limit to the planar rarefaction wave for the
  two-dimensional compressible {N}avier-{S}tokes equations.
\newblock {\em Comm. Math. Phys.}, 376(1):353--384, 2020.

\bibitem{LWW-1}
L.~A. Li, T.~Wang, and Y.~Wang.
\newblock Stability of planar rarefaction wave to 3{D} full compressible
  {N}avier-{S}tokes equations.
\newblock {\em Arch. Rational Mech. Anal.}, 230:911--937, 2018.

\bibitem{LW}
L.~A. Li and Y.~Wang.
\newblock Stability of the planar rarefaction wave to the two-dimensional
  compressible {N}avier-{S}tokes equations.
\newblock {\em SIAM J. Math. Anal.}, 50:4937--4963, 2018.

\bibitem{Lions_book1}
P.-L. Lions.
\newblock {Mathematical topics in fluid mechanics. Vol. 1. Incompressible
  models. Oxford Lecture Series in Mathematics and its Applications, 3. Oxford
  Science Publications. The Clarendon Press, Oxford University Press, New
  York}.
\newblock 1998.

\bibitem{Serre_book}
D.~Serre.
\newblock {\em Systems of conservation laws I, II}.
\newblock Cambridge University Press, Cambridge, 1999.

\bibitem{Serre-Vasseur}
D.~Serre and A.~Vasseur.
\newblock {$L^2$}-type contraction for systems of conservation laws.
\newblock {\em J. \'Ec. polytech. Math.}, 1:1--28, 2014.

\bibitem{SV_16}
D.~Serre and A.~Vasseur.
\newblock About the relative entropy method for hyperbolic systems of
  conservation laws.
\newblock {\em Contemp. Math. AMS}, 658:237--248, 2016.

\bibitem{SV_16dcds}
D.~Serre and A.~Vasseur.
\newblock The relative entropy method for the stability of intermediate shock
  waves; the rich case.
\newblock {\em Discrete Contin. Dyn. Syst.}, 36(8):4569--4577, 2016.

\bibitem{Stokols}
L.~Stokols.
\newblock {$L^2$-type contraction of viscous shocks for large family of scalar
  conservation laws}.
\newblock {\em J. Hyperbolic Differ. Equ., To appear}.

\bibitem{Vasseur_gamma3}
A.~Vasseur.
\newblock Time regularity for the system of isentropic gas dynamics with
  {$\gamma=3$}.
\newblock {\em Comm. Partial Differential Equations}, 24(11-12):1987--1997,
  1999.

\bibitem{Vasseur-2013}
A.~Vasseur.
\newblock Relative entropy and contraction for extremal shocks of conservation
  laws up to a shift.
\newblock In {\em Recent advances in partial differential equations and
  applications}, volume 666 of {\em Contemp. Math.}, pages 385--404. Amer.
  Math. Soc., Providence, RI, 2016.

\bibitem{VW}
A.~Vasseur and Y.~Wang.
\newblock The inviscid limit to a contact discontinuity for the compressible
  navier-stokes-fourier system using the relative entropy method.
\newblock {\em SIAM J. Math. Anal.}, 47(6):4350--4359, 2015.

\end{thebibliography}
\end{document}